\documentclass[aop]{imsart}

\RequirePackage{amsthm,amsmath,amsfonts,amssymb}
\RequirePackage[numbers]{natbib}

\usepackage{epsfig}
\usepackage[usenames]{color}
\usepackage{verbatim}
\usepackage{hyperref}
\usepackage{multicol}
\usepackage{comment}
\usepackage{float}
\usepackage{graphicx}
\usepackage{centernot}
\usepackage{tikz}
\usepackage{color}
\usepackage[]{algorithm2e}
\usepackage{enumerate}

\usepackage{pgfplots}
 
\pgfplotsset{compat = newest}

\graphicspath{{./Pics}}

\usepackage[colorinlistoftodos, color=green]{todonotes}
\usepackage[normalem]{ulem}

\startlocaldefs

\theoremstyle{plain}
\newtheorem{theorem}{Theorem}[section]
\newtheorem{lemma}[theorem]{Lemma}
\newtheorem{proposition}[theorem]{Proposition}

\theoremstyle{remark}
\newtheorem{definition}[theorem]{Definition}

\numberwithin{equation}{section}
\newtheorem{remark}[theorem]{Remark}

\newtheorem{question}[theorem]{Question}

\numberwithin{equation}{section}

\newcommand{\1}{{\text{\Large $\mathfrak 1$}}}

\newcommand{\be}{\begin{equation}}
\newcommand{\ee}{\end{equation}}

\providecommand{\norm}[1]{\Vert#1\Vert}
\providecommand{\normO}[1]{\vert#1\vert}

\newcommand{\eone}{\textup{e}_1}

\makeatletter
\let\@fnsymbol\@arabic
\makeatother

\def\bE{\mathbb{E}}
\def\bN{\mathbb{N}}
\def\bP{\mathbb{P}}

\def\bR{\mathbb{R}}
\def\bZ{\mathbb{Z}}

\def\cL{\mathcal{L}}

\def\cU{\mathcal{U}}
\def\cP{\mathcal{P}}

\def\c{\textup{c}} 

 \def\Z{\bZ} 
 
\def\R{\bR}
\def\N{\bN}
\def\P{\bP} 
\def\cO{\mathcal{O}}

\newcommand{\pr}[1]{\mathbb{P}\!\left(#1\right)}
\newcommand{\Ex}[1]{\mathbb{E}\!\left[#1\right]}
\newcommand{\estart}[2]{\mathbb{E}_{#2}\!\left[#1\right]}

\newcommand{\prcond}[3]{\mathbb{P}_{#3}\!\left(#1\;\middle\vert\;#2\right)}
\newcommand{\econd}[2]{\mathbb{E}\!\left[#1\;\middle\vert\;#2\right]}
\newcommand{\escond}[3]{\mathbb{E}_{#3}\!\left[#1\;\middle\vert\;#2\right]}

\newcommand{\til}{\widetilde}

\renewcommand{\epsilon}{\varepsilon}
\usepackage{stackengine}

\newcommand{\Er}{E}

\newcommand{\dif}{\textup{d}}

\DeclareMathOperator{\Var}{Var}

\def\E{\bE}
\def\P{\bP} 

\definecolor{partcolor1}{rgb}{0.0,0.5,0.0}
\definecolor{partcolor2}{rgb}{0.0,0.5,0.0}

\definecolor{darkgreen}{rgb}{0.0,0.5,0.0}
\definecolor{darkblue}{rgb}{0.5,0.1,0.5}
\definecolor{deepblue}{rgb}{0.25,0.41,0.88}
\definecolor{nicosred}{rgb}{0.65,0.1,0.1}
\definecolor{light-gray}{gray}{0.7}
\allowdisplaybreaks[1]

\endlocaldefs

\begin{document}

\begin{frontmatter}
\title{Biased random walk on dynamical percolation}
\runtitle{Biased random walk on dynamical percolation}

\begin{aug}
\author[A]{\fnms{Sebastian}~\snm{Andres}\ead[label=e1]{sebastian.andres@tu-braunschweig.de}}
\author[B]{\fnms{Nina}~\snm{Gantert}\ead[label=e2]{ gantert@ma.tum.de }},
\author[C]{\fnms{Dominik}~\snm{Schmid}\ead[label=e3]{d.schmid@uni-bonn.de}},
\and
\author[D]{\fnms{Perla}~\snm{Sousi}\ead[label=e4]{p.sousi@statslab.cam.ac.uk}}
\address[A]{Institut f\"ur Mathematische Stochastik, Technische Universit\"at Braunschweig,
\printead{e1}}

\address[B]{Department of Mathematics, Technical University of Munich,
\printead{e2}}

\address[C]{Institute for Applied Mathematics,  University of Bonn,
\printead{e3}}

\address[D]{Department of Pure Mathematics and Mathematical Statistics,
University of Cambridge,
\printead{e4}}

\end{aug}

\begin{abstract}
We study biased random walks on dynamical percolation on $\Z^d$. We establish a law of large numbers and an invariance principle for the random walk using regeneration times. Moreover, we verify that the Einstein relation holds, and we investigate the speed of the walk as a function of the bias. While for $d=1$ the speed is increasing, we show that in general this fails in dimension $d \geq 2$. As our main result, we establish two regimes of parameters, separated by an explicit critical curve, such that the speed is either eventually strictly increasing or eventually strictly decreasing. This is in sharp contrast to the biased random walk on a static supercritical percolation cluster, where the speed is known to be eventually zero.
\end{abstract}



\end{frontmatter}


\section{Introduction} \label{sec:Introduction}

In this paper, we introduce and study biased random walks in dynamically evolving environments. 
The model of random walks on dynamical percolation was introduced in \cite{PSS:DynamicalPercolation} by Peres, Stauffer and Steif, and has the following description. 

Fix a locally finite graph $G=(V,E)$ and an initial  state $\eta \in \{0,1\}^{E(G)}$ of the edges.  For parameters $\mu > 0$ and~$p \in [0,1]$, we consider the dynamics $(\eta_t)_{t \geq 0}$ with $\eta_0=\eta$, where each edge $e$ in the graph is assigned an independent Poisson process of rate~$\mu$. If there is a point of the Poisson process at time $t$, we refresh the state of $e$ in $\eta_t$, i.e.\ we declare $e$ open (and write $\eta_t(e)=1$) with probability $p$ and closed (and write $\eta_t(e)=0$) with probability~$1-p$, independently of all other edges and previous states of $e$. We then say an edge $e$ is \textbf{open} at time $t$ if $\eta_t(e)=1$, and \textbf{closed} otherwise.

From now on, we focus on the case where the underlying graph is $\Z^d$ with $d\geq 1$. We define a continuous-time random walk $(X_t)_{t \geq 0}$ in the environment $(\eta_t)_{t \geq 0}$ with bias parameter $\lambda>0$ as follows: set $X_0=0$ and assign a rate $1$ Poisson clock to the particle. We also set for $\lambda>0$,
\begin{align}\label{eq:defz}
Z_\lambda:= e^\lambda + e^{-\lambda}  + 2d-2.	
\end{align}
Whenever the clock rings at time $t$ and the random walker is currently at a site $x$, we choose one of the neighbours $y$ of $x$ with probability
\begin{align*}
p(x,x \pm \textup{e}_i) &= \frac{1}{Z_\lambda} \qquad \text{ for } i \in \{2,\dots,d\} , \\
 p(x,x + \eone) &= \frac{\textup{e}^\lambda}{Z_\lambda}, \qquad p(x,x - \eone) = \frac{\textup{e}^{-\lambda}}{Z_\lambda}\, .
\end{align*} 
If $\eta_t(\{x,y\})=1$, the random walker moves from $x$ to $y$, and it stays at $x$, otherwise. 
We will call the process $(X_t,\eta_t)_{t\geq0}$ a $\lambda$-biased random walk on dynamical percolation with parameters~$\mu$ and~$p$.  Throughout this paper, we say that an edge $e$ is \textbf{examined} by the walker at time $t$ if the clock for the walker rings at time $t$ and the walker chooses the edge $e$ for a possible jump. 

Note that $(\eta_t)_{t \geq 0}$ and $(X_t,\eta_t)_{t \geq 0}$ are  Markov processes, while $(X_t)_{t \geq 0}$ is not. Moreover, $(\eta_t)_{t \geq 0}$ has the Bernoulli-$p$-product measure $\pi_p$ on $\{0,1\}^{E(\Z^d)}$ as its unique invariant distribution, and we assume in the following that $\eta=\eta_0 \sim\pi_p$. 

In this paper our focus is on the speed of the first coordinate of the walk as a function of the bias. The motivation to study this question comes from the two different regimes  one observes in the case of a biased random walk on a static percolation cluster. It was first shown in~\cite{BGP:SpeedPercolation} and~\cite{S:Anisotropic} that when $p>p_c$ and $X$ is a $\lambda$-biased random walk on the infinite percolation cluster, then there exist $\lambda_1<\lambda_2$ so that when $\lambda>\lambda_2$, the speed is $0$, while for $\lambda<\lambda_1$, the speed is strictly positive. A few years later it was proved by~\cite{FH:SpeedPercolation} that there is a sharp transition, i.e.\ there exists $\lambda_*$ so that for all $\lambda>\lambda_*$ the speed is equal to $0$, while for $\lambda<\lambda_*$ the speed is strictly positive. Motivated by these results, in this paper we study the speed in the dynamical setting and we establish that for all choices of the parameters, the speed is always strictly positive and it satisfies an Einstein relation as we show in Theorem~\ref{thm:main} below. Our second main result concerns the monotonicity of the speed as a function of the bias in dimensions $d\geq 2$, where we observe two different regimes.

Before stating our results we recall an invariance principle established in~\cite[Theorem~3.1]{PSS:DynamicalPercolation} in the unbiased case. Unless otherwise stated, our probability measure is taking averages not only over the walk but over the environment as well.

\begin{theorem}[{\cite[Theorem 3.1]{PSS:DynamicalPercolation}}]\label{thm:pssinvariance} For $d\geq 1$, $\mu>0$,  $p\in (0,1)$ and $\lambda=0$, there exists $\sigma\in (0,\infty)$ so that 
\begin{align*}
	\left( \frac{X_{kt}}{\sqrt{k}} \right)_{t \in [0,1]} \overset{(d)}{\rightarrow} (\sigma B_t)_{t \in [0,1]}
\end{align*}
	in $D[0,1]$ as $k\to\infty$, where $(B_t)_{t\geq 0}$ is a standard Brownian motion.
\end{theorem}

We now present our first result on the speed of the biased random walk $(X_t)_{t \geq 0}$ for fixed environment parameters $\mu >0$ and $p \in (0,1)$.

\begin{theorem}\label{thm:main}
	Let $d\geq 1$ and let $(X_t,\eta_t)_{t\geq 0}$ be a $\lambda$-biased random walk on dynamical percolation on $\Z^d$ with parameters $\mu>0$ and $p\in (0,1)$. Then for all $\lambda>0$, there exists $v(\lambda)=v_{\mu,p}(\lambda)$ such that almost surely
\begin{equation}\label{eq:SpeedBRW}
\lim_{t \rightarrow \infty} \frac{X_t}{t} = (v(\lambda),0,\dots,0) \, .
\end{equation} 
Moreover, the function $\lambda\mapsto v(\lambda)$ is strictly positive for all $\lambda>0$, continuously differentiable and satisfies
\begin{align}\label{einstein-a}
	\lim_{\lambda \rightarrow 0} v^{\prime}(\lambda)=\sigma^2,
\end{align}
where $\sigma$ is as in Theorem~\ref{thm:pssinvariance}.
\end{theorem}

The last statement in the theorem above is known as \textbf{Einstein relation}. Moreover, as we will see in Proposition~\ref{pro:Invariance}, an invariance principle also holds in the biased case and the proof follows along the same lines as the proof of Theorem~3.1 in~\cite{PSS:DynamicalPercolation}.

When $d=1$, using the obvious coupling between two walks with different bias parameters, it is immediate to see that the speed is always monotone increasing in the bias and in fact in  Section~\ref{sec:SpeedRWOneDim} we also establish that in $d=1$ the speed is strictly increasing as a function of the bias. It is thus natural to ask what happens for $d\geq 2$. While the speed turns out to be monotone increasing in~$\lambda>0$ for certain regimes of $\mu>0$ and $p\in (0,1)$ in dimensions $d\geq 2$ as we show in Section~\ref{sec:SpeedRWTwoDim}, our main result is an explicit criterion deciding whether the speed is eventually strictly increasing or decreasing.

\begin{theorem}[Monotonicity of the speed for $d \geq2$]\label{thm:mainTwoDim}Consider the biased random walk on dynamical percolation on $\Z^d$ for $d\geq 2$. For all $p \in (0,1)$ and $\mu>0$, there exists some $\lambda_0=\lambda_0(\mu,d)$ such that the following hold.
\begin{itemize}
\item[\rm(1)] The speed $v(\lambda)$ is strictly increasing for all $\lambda \geq \lambda_0$ provided that $\mu^2 > p(1-p)$.
\item[\rm(2)] The speed $v(\lambda)$ is strictly decreasing for all $\lambda \geq \lambda_0$ provided that $\mu^2 < p(1-p)$.
\end{itemize}
\end{theorem}

\begin{remark}\rm{
Note that this is in contrast to the biased random walk on a static super-critical percolation cluster, where the speed is known to be zero for large values of $\lambda$; see  \cite{BGP:SpeedPercolation,FH:SpeedPercolation,S:Anisotropic}.  
The criterion for the eventual monotonicity of the speed, identified in Theorem~\ref{thm:mainTwoDim}, is visualised in Figure~\ref{fig:EventuallyMonotone}. Moreover, this suggests different shapes of the speed functions; see Figure~\ref{fig:Shapes}.} 
\end{remark}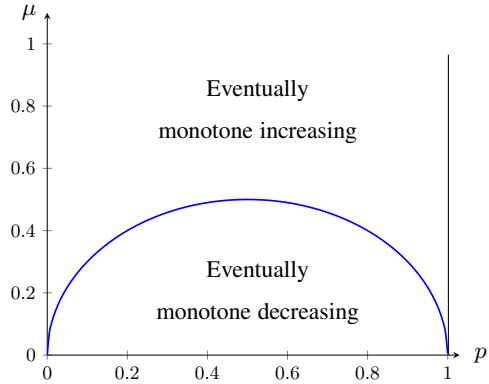
\begin{figure}\centering
\begin{tikzpicture}[scale=0.8]
\draw[thin](6.666,5) to (6.666,0);  

\begin{axis}[
    xmin = 0, xmax = 1.03,
    ymin = 0, ymax = 1.1, axis lines=left]
    \addplot[domain = 0:1,
        samples = 200,
        smooth,
        thick,
        blue,]{sqrt(x*(1-x)) )};
\end{axis}
\node[scale=1] (ab) at (3.5,1.4) {Eventually};
\node[scale=1] (ab) at (3.5,4.4) {Eventually};
\node[scale=1] (ab) at (3.5,0.7) {monotone decreasing};
\node[scale=1] (ab) at (3.5,3.7) {monotone increasing};

\node[scale=1] (ab) at (7.2,0) {$p$};

\node[scale=1] (ab) at (-0.3,5.7) {$\mu$};

\end{tikzpicture}
\caption{\label{fig:EventuallyMonotone}Plot of the different regimes in Theorem~\ref{thm:mainTwoDim} for large $\lambda$.}
\end{figure}
\subsection{Related work} 
Biased random walks in random media were investigated intensively over the last years, we refer to \cite{A:speed, AH09, BD:Percolation, BFS14, BGN:BiasedRWconductances, BGP:SpeedPercolation, BMT23, Bo19, Cr13, D:Percolation, DS:Trapping,   FGS18, FH:SpeedPercolation, FP18, GMM18, LPP:BiasedRW, S:Anisotropic} for a non-exhaustive list and to \cite{BF16} for a survey.
The most prominent examples are biased random walks on Galton-Watson trees  and biased random walks on supercritical percolation clusters, see~\cite{LPP:BiasedRW} and~\cite{BGP:SpeedPercolation, S:Anisotropic} respectively.
Our work is in particular motivated by the study of biased random walks on (static) supercritical percolation clusters. This model was introduced in 
\cite{BD:Percolation}. Due to traps in the cluster, the speed of the walk is zero for large values of the bias.  Simulations indicate that the speed is a unimodal function of the bias, first increasing until the maximum is achieved and then decreasing and eventually becoming zero, see \cite{DS:Trapping}. Indeed, in the breakthrough paper~\cite{FH:SpeedPercolation}, it was shown that there is a critical value separating the positive speed regime from the zero speed regime.
In the case of Galton-Watson trees with leaves, this phase transition was known earlier and there is an explicit formula for the critical value, see~\cite{LPP:BiasedRW}.
When there are no ``hard traps", the speed of the biased walk should be strictly positive, and one may ask if it is increasing as a function of the bias.
In the case of a Galton-Watson tree without leaves, it is conjectured that the speed is indeed an increasing function of the bias.
While this conjecture still remains open, it was proved that the speed is eventually increasing, see \cite{BFS14} and~\cite{A:speed}.
The same argument as in \cite{BFS14} gives that for biased random walk among uniformly elliptic i.i.d.\ conductances, the speed is eventually increasing as a function of the bias. However, it was shown in~\cite{BGN:BiasedRWconductances} that for some laws of the conductances, the speed is not increasing for all values of the bias, i.e.\ there exist~$\lambda_1 < \lambda _2$ such that $v(\lambda_1) > v(\lambda_2)$.\\
In the presence of hard traps, a central limit theorem is expected to hold for small values of the bias, in a strict subset of the positive speed regime.
This was proved for biased random walk on supercritical percolation clusters in \cite{S:Anisotropic} and \cite{BGP:SpeedPercolation} and for random walks on Galton-Watson trees with leaves in \cite{LPP:BiasedRW}. For other models such results have been established for example in~\cite{GMM18} and~\cite{Bo19}.
In an environment without hard traps, there are examples where a central limit theorem holds for all values of the bias, see for instance~\cite{PZ08, Sh02}.
In our case, this also turns out to be true, see Proposition~\ref{pro:Invariance}.

The literature on (unbiased) random walks in time-dependent random environments is too vast to give a review, we just point to two papers which are relevant in our setup, namely
\cite{BZ:InvarianceQuenched, BR:TimeDep}, see also the references therein.
Unbiased random walks on dynamical percolation have been studied in particular in terms of their mixing times, see \cite{HS:ComparisonDynamicalPercolation, PSS:QuenchedMixing, PSS:MixingSupercriticialPercolation, PSS:DynamicalPercolation, ST:CutoffErdos}.

\subsection{Overview of proof ideas}
The proof of the first part of Theorem~\ref{thm:main} follows by using a suitably defined sequence of regeneration times $(\tau_i)$, i.e.\ a sequence of random times such that the evolution of the walk and the environment between $[\tau_i,\tau_{i+1}]$ is independent for different choices of $i$. A key property of the regeneration times that we define is that their distribution only depends on the parameter $\mu$, but not on the bias parameter $\lambda$ and the percolation parameter $p$. Using these regeneration times and a law of large numbers we get the existence of the speed.  Moreover, we find an expression for the speed in terms of an infinite series which allows to give a simple expression for its derivative; see Lemma~\ref{lem:Differentiability}. In particular, the speed is strictly increasing for all~$\lambda$ sufficiently small.

We now give an overview of the main ideas behind the proof of our main result, Theorem~\ref{thm:mainTwoDim}, giving a necessary and sufficient condition for the speed to be eventually in $\lambda$ strictly increasing or decreasing. There are two main ingredients. First, in Proposition~\ref{pro:MonotoneSubsequence}, we obtain an asymptotic expression for the speed which is valid for all bias parameters $\lambda$ sufficiently large. Next, in Lemma~\ref{lem:ApproximateDerivate}, we give an asymptotic bound on the derivative of the speed for large $\lambda$. The proof of Theorem~\ref{thm:mainTwoDim} is a direct consequence of these two results. 

In order to prove Proposition~\ref{pro:MonotoneSubsequence}, we start with a detailed analysis of the speed in the one-dimensional case. In order to analyse the case $d\geq 2$, we rely on the regeneration times and compare the first coordinate of the walk with a time-changed one-dimensional walk in a suitably defined evolving environment. To be more precise, we construct a coupling which keeps the first coordinates of the two walks together until the second time that the $d$-dimensional walk jumps in a direction other than $\eone$. 

In order to prove Lemma~\ref{lem:ApproximateDerivate} we rely on a comparison between walkers with different bias parameters using marked Poisson point processes. A key task is to develop an asymptotic expression for the derivative of the speed on the scale $e^{-\lambda}$, with the constants only depending on $\mu$ and $p$. 


\begin{figure}\centering
\begin{tikzpicture}[scale=0.6]

   \draw[->] (-0.3, 0) -- (6, 0) node[right] {$\lambda$};
  \draw[->] (0, -0.3) -- (0, 4) node[above] {$v(\lambda)$};
  \draw[scale=1, domain=0:6, smooth, variable=\x, blue] plot ({\x}, {(2*\x*\x*\x*\x-5*\x*\x+8*\x)/(\x*\x*\x*\x+1)});
    \draw[ domain=0:6, smooth, variable=\y, densely dotted, thick]  plot ({\y}, {1.95});
 
  \draw[->] (-0.3+8, 0) -- (6+8, 0) node[right] {$\lambda$};
  \draw[->] (0+8, -0.3) -- (0+8, 4) node[above] {$v(\lambda)$};
  \draw[scale=1, domain=0:6, smooth, variable=\x, blue] plot ({8+\x}, {(2*\x*\x*\x+5*\x)/(\x*\x*\x+1)});
  \draw[ domain=0:6, smooth, variable=\y, densely dotted, thick]  plot ({\y+8}, {2}); 

  \draw[->] (-0.3+8+8, 0) -- (6+8+8, 0) node[right] {$\lambda$};
  \draw[->] (0+8+8, -0.3) -- (0+8+8, 4) node[above] {$v(\lambda)$};
  \draw[scale=1, domain=0:6, smooth, variable=\x, blue] plot ({\x+16}, {(2*\x*\x*\x*\x-2*\x*\x+1.6*\x)/(\x*\x*\x*\x+1)});
    \draw[ domain=0:6, smooth, variable=\y, densely dotted, thick]  plot ({\y+16}, {2});
\end{tikzpicture}
\caption{\label{fig:Shapes}The different pictures show three possible shapes for $v(\lambda)$ which are in accordance  with Theorem \ref{thm:mainTwoDim}.}
\end{figure}
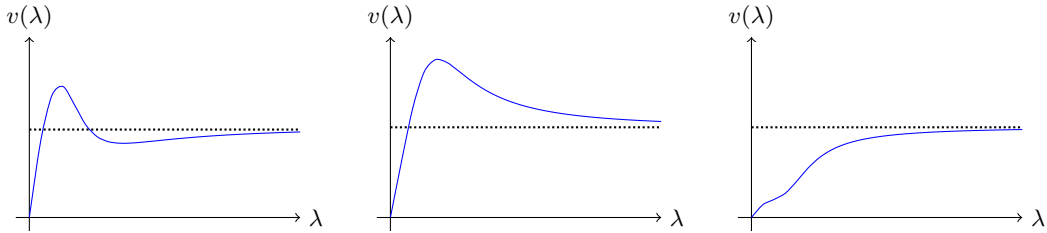

\subsection{Organisation}
In Section~\ref{sec.RegenerationTimes} we define our sequence of regeneration times that will be used in the rest of the paper. In Section~\ref{sec:SpeedRWSpeed} we prove Theorem~\ref{thm:main}. In Section~\ref{sec:monotonicity} we prove Theorem~\ref{thm:mainTwoDim} and we also study the one-dimensional case in Section~\ref{sec:SpeedRWOneDim} where we establish the strict monotonicity of the speed for all values of $\lambda$. Finally, in Section~\ref{sec:SpeedRWTwoDim} we prove that the speed in $d\geq 2$ is strictly increasing when $\mu$ is sufficiently large, or~$p$ is sufficiently close to $1$.

\subsection{Open problems}\label{sec:OpenProblems}

We finish this section by stating a couple of open questions. The first question concerns the monotonicity of the speed for general directions of the drift.

\begin{question} Consider the situation where the bias of the random walk in dynamical percolation on $\Z^d$ for $d\geq 2$ is not along one of the axes $\pm \textup{e}_i$ with some $i \in \{1,\dots, d\}$. In analogy to Theorem \ref{thm:mainTwoDim}, is there an explicit criterion to decide whether the speed is eventually monotone increasing in $\lambda$?
\end{question}

The second open question concerns the monotonicity of the speed along the critical curve established in Theorem \ref{thm:mainTwoDim}. 

\begin{question} For $\mu^2=p(1-p)$ is the speed eventually monotone increasing in $\lambda>0$?
\end{question}

The last open question concerns the case whether an eventually monotone increasing speed implies that the speed is monotone in $\lambda>0$.

\begin{question} Is there a choice of parameters $\mu>0$ and $p \in (0,1)$ with $\mu^2>p(1-p)$ such that $\lambda \mapsto v(\lambda)$ is not monotone increasing in $\lambda>0$?
\end{question}

%

%

\section{Regeneration times for the biased random walk}\label{sec.RegenerationTimes}

Throughout the paper we write $\mathbb{P}^{p,\mu}$ for the probability measure corresponding to dynamical percolation with parameters~$\mu$ and $p$. For $\lambda\geq 0$, we write $\mathbb{P}_\lambda$ for the semi-direct product of $\mathbb{P}^{p,\mu}$ with the law of the $\lambda$-biased random walk starting at the origin. We write $\mathbb{E}_\lambda$ for the expectation with respect to $\mathbb{P}_\lambda$.

In order to prove Theorem \ref{thm:main} we need to define a sequence of regeneration times for the random walk on dynamical percolation. A similar definition of regeneration times was used by Peres, Stauffer and Steif in~\cite{PSS:DynamicalPercolation}. Here we use the definition given in~\cite{HS:ComparisonDynamicalPercolation} which works for general underlying graphs and was used in~\cite{HS:ComparisonDynamicalPercolation} to compare mixing and hitting times for random walks on dynamical percolation in terms of the respective quantities for the static graph. For the biased random walk on $\Z^d$, we have the following construction following~\cite[Section~3]{HS:ComparisonDynamicalPercolation}.

\begin{figure}\centering
\begin{tikzpicture}[scale=0.35]

\pgfmathsetseed{1332}

\foreach \x in{1,...,15}{
	\draw[gray!50,thin](1,\x) to (15,\x);  
	\draw[gray!50,thin](\x,1) to (\x,15);   
}
	
\foreach \x in{1,...,15}{	
\foreach \y in{1,...,14}{

     \pgfmathparse{rnd}
    \let\dummynum=\pgfmathresult
    \ifdim\pgfmathresult pt < 0.5 pt\relax \draw[red,line width=1pt] (\x,\y) --(\x,\y+1);\fi

}
	}

\foreach \x in{1,...,14}{	
\foreach \y in{1,...,15}{ 
 
     \pgfmathparse{rnd}
    \let\dummynum=\pgfmathresult
    \ifdim\pgfmathresult pt < 0.5 pt\relax \draw[red,line width=1pt] (\x,\y) --(\x+1,\y);\fi    
    
}
	}

	\draw[fill,black](6.8+3,7.8+2) rectangle (7.2+3,8.2+2);   

\node[scale=1] (ab) at (6.5+2.7,8.5+2) {$X_t$};

\def\z{18};

\foreach \x in{1,...,15}{
	\draw[gray!50,thin](1+\z,\x) to (15+\z,\x);  
	\draw[gray!50,thin](\x+\z,1) to (\x+\z,15);   
}
	
\foreach \x in{1,...,15}{	
\foreach \y in{1,...,14}{

     \pgfmathparse{rnd}
    \let\dummynum=\pgfmathresult
    \ifdim\pgfmathresult pt < 0.5 pt\relax \draw[red!50,line width=1pt,densely dotted] (\x+\z,\y) --(\x+\z,\y+1);\fi

}
	}

\foreach \x in{1,...,14}{	
\foreach \y in{1,...,15}{ 
 
     \pgfmathparse{rnd}
    \let\dummynum=\pgfmathresult
    \ifdim\pgfmathresult pt < 0.5 pt\relax \draw[red!50,line width=1pt,densely dotted] (\x+\z,\y) --(\x+1+\z,\y);\fi    
    
}
	}

\draw[blue,line width=1.5pt] (7+\z,8)--++(-1,0);
\draw[blue,line width=1.5pt] (7+\z,8)--++(0,1);
\draw[blue,line width=1.5pt] (6+\z,8)--++(-1,0);
\draw[blue,line width=1.5pt] (6+\z,7)--++(0,-1);

\draw[blue,line width=1.5pt] (3+\z,8)--++(1,0);
\draw[blue,line width=1.5pt] (8+\z,8)--++(0,1);
\draw[blue,line width=1.5pt] (6+\z,5)--++(1,0);

\draw[red,line width=1.5pt] (7+\z,8)--++(2,0);
\draw[red,line width=1.5pt] (7+\z,7)--++(1,0);
\draw[red,line width=1.5pt] (5+\z,7)--++(1,0);
\draw[red,line width=1.5pt] (8+\z,8)--++(0,-2);
\draw[red,line width=1.5pt] (5+\z,6)--++(1,0);

\draw[red,line width=1.5pt] (5+\z,8)--++(0,1);

\draw[fill,black](6.8+\z,7.8) rectangle (7.2+\z,8.2);   
	
\node[scale=1] (ab) at (6.3+\z,8.6) {$X_t$};

\end{tikzpicture}
\caption{\label{fig:Dynamic} Visualisation of the dynamical percolation cluster and the infected set $I_t$.
The edges in red on the left correspond to open edges. In the right picture, blue edges correspond to closed edges, which are infected at time $t$, red edges correspond to open edges that have a copy in $I_t$ and dotted red edges correspond to open edges that do not have a copy in $I_t$.}
\end{figure}
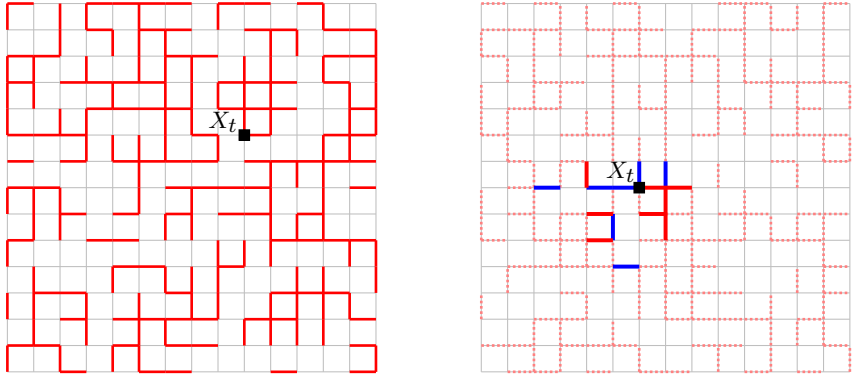

We fix an enumeration $(e_i)_{i \in \N}$ of the edges in $\Er$ according to an arbitrary rule. Then for each edge~$e_i$, we create an infinite number of copies denoted $e_{i,1}, e_{i,2}, \dots$. We now define a process $(I_t)_{t \geq 0}$, where for every $t\geq 0$, $I_t$ is a set containing copies of edges that we refer to as the \textbf{infected set}. \footnote{Note that in~\cite{HS:ComparisonDynamicalPercolation} the infected set can contain also the original edges.} 
Let $I_0=\emptyset$. Suppose that for some $t \geq 0$, the Poisson clock associated with the random walk $(X_t)_{t \geq 0}$ rings and that the walker examines the edge $e_i$ for some $i\in \N$. If no copy of  $e_i$ is contained in $I_{t_-}$, we set
\begin{equation}
I_{t} := I_{t_-} \cup \{e_{i,1} \} \, .
\end{equation} Otherwise, we add to $I_t$ the copy $e_{i,j}$ of $e_i$ with the smallest index $j$ such that $e_{i,j} \notin I_{t_-}$. 

Next, for all $t \geq 0$, we assign the lexicographic ordering $\preceq$ to the edges in $I_t$ using the ordering of the edges of $E$, i.e.\ for $e_{i,j},e_{k,l} \in I_t$ we have
\begin{equation}
e_{i,j} \preceq e_{k,l} \quad  \Leftrightarrow \quad (i \leq k) \vee ((i=k) \wedge (j \leq l))
\end{equation}
Further, let $(N_t)_{t \geq 0}$ be a Poisson process with time dependent intensity $\mu |I_t|$. Whenever a clock of this process rings at time $t$, we choose an index uniformly at random from $\{1,\dots,|I_t|\}$ and remove the copy of the edge with this index in $I_t$ according to the ordering $\preceq$. Moreover, if the picked copy is of the form $e_{i,1}$ for some $i\in \N$, we refresh the state of the edge $e_i$ in the environment $\eta_t$, i.e.\ we set $\eta_t(e_i)=1$ with probability $p$, and $\eta_t(e_i)=0$, otherwise. 

For all edges $e_j$ with $e_{j,1} \notin I_t$, we use independent rate $\mu$ Poisson clocks to determine when the state of the edge in $(\eta_t)_{t \geq 0}$ is refreshed. Note that with this construction $(X_t,\eta_t)_{t \geq 0}$ has indeed the correct transition rates, see Figure~\ref{fig:Dynamic} for an illustration.

Recall that we start from $\eta_0 \sim \pi_p$, $X_0=0$, and that we set $I_0=\emptyset$. Let $\tau_0:=0$. For every $i\in \N$, we set
\begin{equation}\label{def:RegenTimes}
\tau_{i+1}:= \inf\{ t > \tau_{i} \colon I_{t} = \emptyset \text{ and } I_{t^{\prime}} \neq \emptyset \text{ for some } t^{\prime} \in (\tau_i,t)  \} \, .
\end{equation}
Let $\N_0:= \N \cup \{0\}$ and note that the times $(\tau_i)_{i \in \N_0}$ are indeed \textbf{regeneration times} for the process $(X_t)_{t \geq 0}$, i.e.\ $(\tau_{i}-\tau_{i-1})_{i \geq 1}$ are i.i.d.\ and the random walk increments $(X_{\tau_i}-X_{\tau_{i-1}})_{i \geq 1}$ are i.i.d. 

\begin{remark}
Let us stress that the law of $(\tau_i)_{i \in \N_0}$ only depends on the parameter $\mu$ and not on $\lambda$ or $p$. Further, in contrast to many standard constructions of regeneration times, we have here the advantage that the regeneration times are also finite in the case $\lambda=0$.
\end{remark}

Observe that the process $(|I_t|)_{t \geq 0}$ is a continuous-time birth-and-death chain on $\N_0$ with transition rates $q(i-1,i)=1$ and $q(i,i-1)=\mu i$ for all $i\in \N$. The following lemma is the content of~{\cite[Lemma~3.5]{HS:ComparisonDynamicalPercolation}}.

\begin{lemma}\label{lem:ExpoTailsTau} For all $p\in (0,1)$, for all $\lambda\geq 0$, and for all $\mu>0$, the increments $(\tau_{i}-\tau_{i-1})_{i \in \N}$ are i.i.d., have exponential tails and satisfy $\estart{\tau_1}{\lambda} = e^{1/\mu}$.

\end{lemma}

 Therefore, with a slight abuse of notation, we will write $\P$ instead of $\P_{\lambda}$ when considering events involving only the regeneration times. 
For every $t\geq 0$ we let $\cU_a(t)$ be the number of attempted jumps of the walker $X$ up to time $t$, which follows the Poisson distribution with parameter $t$. We have the following result on 
$\cU_a(\tau_1)$.

\begin{lemma}\label{cor:exptailsofutauone}
	For every $\mu>0$ and $p\in (0,1)$, there exists a positive constant $c_{\mu}$ satisfying $c_\mu\to \infty$ as $\mu\to\infty$ so that for all~$m\geq 2$ 
		\[
	\pr{\cU_a(\tau_1)\geq m} \leq e^{-c_{\mu} m}.
	\]
\end{lemma}

\begin{proof}[\bf Proof]
The fact that the random variable $\cU_a(\tau_1)$ has exponential tails is an immediate consequence of Lemma~\ref{lem:ExpoTailsTau} and the exponential concentration of a Poisson random variable around its mean. 
It remains to show that we can choose $c_{\mu}$ such that $c_\mu\to \infty$ as $\mu\to\infty$. Let $\mu>1$. Recall the birth and death chain $(|I_t|)_{t \geq 0}$, and let $(S_k)_{k\geq 0}$ with $S_0=0$ denote its jump chain. Further, let 
\begin{equation}
\tilde{\tau}_0 := \inf\{ n \geq 1	 \colon  S_n=0 \}
\end{equation}
be the  first return time of $(S_k)_{k\geq 0}$ to the origin and observe that $2\cU_a(\tau_1)=\tilde{\tau}_0$. Note that the process $(|I_t|)_{t \geq 0}$ 
is dominated from above by a biased random walk on $\{0,1,\dots\}$ with transition rates $q(i,i-1)= \mu$ and $q(i-1,i)=1$ for all $i\in \N$. Hence, we get for all~$\theta> 0$ that the process $(M_k)_{k \in \N}$ defined by
\[
M_k := e^{\theta S_k} \cdot  f(\theta)^{k-1}
\]
 with $f(\theta):=(\mu+1)/(e^\theta + e^{-\theta}\mu)$ is a super-martingale. Since almost surely $M_1=e^{\theta}$, and 
$\tilde{\tau}_0$ has exponential tails, we can apply the optional stopping theorem together with Fatou's lemma to obtain
	\begin{align*}
e^{\theta} = \E[M_1] \geq \E[M_{\tilde{\tau}_0}] =	 \estart{\exp(\theta S_{\tilde{\tau}_0}) f(\theta)^{\tilde{\tau}_0-1}}{} = \estart{(f(\theta))^{\tilde{\tau}_0-1}}{} 
	\end{align*}
for all $\theta>0$. Take $\tilde{\theta}=\log \log \mu$ for $\mu>0$ sufficiently large such that $f(\tilde{\theta}) \geq 1$. Then, we get that for all $m\geq 2$ by Markov's inequality
\begin{equation}
\pr{\cU_a(\tau_1)\geq m} = \pr{\tilde{\tau}_0 - 1 \geq 2m-1} \leq \frac{\estart{(f(\tilde{\theta}))^{\tilde{\tau}_0-1}}{} }{f(\tilde{\theta})^{2m-1}} \leq e^{\tilde{\theta}}f(\tilde{\theta})^{-(2m-1)} . 
\end{equation}
Since $f(\tilde{\theta}) \geq \frac{1}{2} \log \mu$ for all $\mu>0$ sufficiently large, we conclude.
\end{proof}

\section{Speed and Einstein relation}\label{sec:SpeedRWSpeed}

 We will now show Theorem \ref{thm:main} in two steps. First, we prove in Proposition \ref{pro:RegenerationSpeed} a law of large numbers for the biased random walk on dynamical percolation. Then in Proposition~\ref{pro:Invariance}, we prove an invariance principle. Both proofs use similar arguments to~\cite[Theorem 3.1]{PSS:DynamicalPercolation}. In Proposition~\ref{pro:Einstein} we show that the speed in the $\eone$-direction is strictly positive.

\begin{proposition}\label{pro:RegenerationSpeed} Recall the sequence of regeneration times $(\tau_i)_{i \in \N}$ from \eqref{def:RegenTimes}. Then, 
$\P_{\lambda}$-almost surely 
\begin{equation}\label{eq:SpeedViaTau}
\lim_{t \rightarrow \infty} \frac{X_t}{t} = (v(\lambda),0,\dots,0) = \frac{\E_{\lambda}[X_{\tau_1}]}{\E[\tau_1]} \, ,
\end{equation} 
and, writing $(X_t)_{t\geq 0}=(X_t^1,\ldots,X_t^d)_{t\geq 0}$, we also have
\begin{equation}\label{eq:SpeedViaTau2}
\lim_{t \rightarrow \infty} \E_{\lambda}\left[\frac{X^1_t}{t}\right] = \frac{\E_{\lambda}[X^1_{\tau_1}]}{\E[\tau_1]} \, .
\end{equation} 
\end{proposition}

\begin{proof}[\bf Proof]

We first show that there exists a positive constant $C$ (depending on $\mu$) so that for all $\lambda>0$, 
\begin{align}\label{eq:boundonnorm1}
\estart{\norm{X_{\tau_1}}_1}{\lambda}\leq C.
\end{align}
Recall that for every $t\geq 0$ we write $\cU_a(t)$ for the total number of times that a copy of an edge was added to the infected set during the time interval $[0,t]$. Then 
\[
\norm{X_{\tau_1}}_1\leq \cU_a(\tau_1).
\]
By Lemma~\ref{cor:exptailsofutauone} we get that $\cU_a(\tau_1)$ has exponential tails, and hence this proves~\eqref{eq:boundonnorm1}.

Since the increments $(X_{\tau_i}-X_{\tau_{i-1}})$ are i.i.d.\ and~\eqref{eq:boundonnorm1} holds, we can apply the strong law of large numbers to obtain that almost surely
\begin{equation}\label{eq:LLNAlongTau}
\lim_{k \rightarrow \infty} \frac{X_{\tau_k}}{k} = \lim_{k \rightarrow \infty} \frac{1}{k} \sum_{i=1}^{k} (X_{\tau_i}-X_{\tau_{i-1}}) = \E_{\lambda}[X_{\tau_1}] \, .
\end{equation} 
To prove a law of large numbers for $(X_t)_{t\geq 0}$, for every $t\geq 0$ writing $k=k(t) = \lfloor t/\estart{\tau_1}{}\rfloor$ we get 
\begin{align*}
	\frac{X_t}{t} = \frac{X_t- X_{k\estart{\tau_1}{}}}{t} + \frac{X_{k\estart{\tau_1}{}} - X_{\tau_k}}{t} + \frac{X_{\tau_k}}{t}.
\end{align*}
It now follows that the first fraction on the right hand side above converges to $0$ as $t\to\infty$ almost surely. 
Using that $\tau_k/k\to \estart{\tau_1}{}$ as $k\to\infty$ almost surely, it follows easily that the second fraction on the right hand side above converges to $0$ almost surely. Finally, using~\eqref{eq:LLNAlongTau} we get that the third fraction converges to $\estart{X_{\tau_1}}{\lambda}/\estart{\tau_1}{}$ almost surely as $t\to\infty$ and this concludes the proof of the almost-sure convergence.
To show \eqref{eq:SpeedViaTau2}, note that $\norm{X_t/ t}_1 \leq U_a(t)/t$ and $(U_a(t)/t)_{t \geq 0}$ is bounded in $L^2(\P_{\lambda})$, implying that $(\norm{X_t/ t}_1)_{t \geq 0}$ is bounded in $L^2(\P_{\lambda})$ and in particular uniformly integrable.
\end{proof}

At this point, let us state some consequences of Proposition \ref{pro:RegenerationSpeed}. 

The next proposition follows in exactly the same way as the proof of Theorem~3.1 in~\cite{PSS:DynamicalPercolation} when $\lambda=0$. The only difference from~\cite{PSS:DynamicalPercolation} is the definition of the regeneration times, but the way they are used for the invariance principle is the same as in~\cite{PSS:DynamicalPercolation}.
 We give a sketch of the proof, as we need the expression for the diffusivity of the Brownian motion in order to establish the Einstein relation in Lemma~\ref{lem:Differentiability}.

\begin{proposition}\label{pro:Invariance} The first component $(X^1_t)_{t\geq 0}$ of the biased random walk $(X_t)_{t\geq 0}=(X_t^1,\ldots,X_t^d)_{t\geq 0}$ on dynamical percolation satisfies an invariance principle
\begin{equation}\label{eq:Invariance}
\left( \frac{X^1_{kt}-v(\lambda) kt}{\sqrt{k}} \right)_{t \in [0,1]} \overset{(d)}{\rightarrow}(\sigma B_t)_{t \in [0,1]}
\end{equation} in $D[0,1]$ as $k \rightarrow \infty$, where $(B_t)_{t \geq 0}$ denotes a standard Brownian motion and $\sigma^2=\sigma^2(d,\mu,p,\lambda)= \Var_\lambda(X^1_{\tau_1})(\E[\tau_1])^{-1}$. 
\end{proposition}

\begin{proof}[Sketch of the proof] Since the arguments are analogous to the ones in Theorem~3.1 in \cite{PSS:DynamicalPercolation}, apart from a centering due to the presence of the bias, we will only outline the key steps of the proof.  By Lemma \ref{lem:ExpoTailsTau} and  a similar tightness argument as in Theorem~4.1 of \cite{S:SlowdownRWRE}, it suffices to consider the convergence in 
 \eqref{eq:Invariance} only for $t=1$. Since $(\tau_n)_{n \in \N}$ is a sequence of regeneration times, we have that as $n\to\infty$,
 \begin{equation}\label{eq:NormalConvergence} 
 \frac{X^1_{\tau_n}-v(\lambda) n\E[\tau_1]}{\sqrt{n \Var(X^1_{\tau_1})}} \overset{(d)}{\rightarrow} \mathcal{N}
 \end{equation} where $ \mathcal{N}$ is standard normal random variable. Next, we define
 \begin{equation}
 \ell(k) := \max\{ \ell \in \N  \colon \tau_{\ell} \leq k\}
 \end{equation} to be the index of the last regeneration time before $k$. A standard renewal argument implies that the ratio $\ell(k)/k$ converges almost surely to $\E[\tau_1]^{-1}$. We write now
 \begin{equation}\label{eq:SplitCLT}
 \frac{X^1_k-v(\lambda)k}{\sqrt{k}}=\frac{X^1_k-X^1_{\tau_{\ell(k)}}}{\sqrt{k}} + \frac{X^1_{\tau_{\ell(k)}}-v(\lambda)\tau_{\ell(k)}}{\sqrt{k}} + \frac{v(\lambda)(\tau_{\ell(k)}-k)}{\sqrt{k}} \, .
 \end{equation}
For all $0<s<t$ we write $\cU_a[s,t]$ for the number of copies of edges added to the infected set between times $s$ and $t$. Note that $\cU_a[s,t]$ is a Poisson random variable of parameter $t-s$.  Then
 \begin{equation}\label{eq:NormalConvergence2} 
|X^1_k - X^1_{\tau_{\ell(k)}} |  \leq  \cU_a[\tau_{\ell(k)},\tau_{\ell(k)+1}] \, .
\end{equation}  
Recall from Lemma~\ref{lem:ExpoTailsTau} that the regeneration times have exponential tails, and in particular finite variance. Using equation (4.9) in \cite{S:SlowdownRWRE} for the second statement, we see 
\begin{equation}\label{eq:NormalConvergence3} 
\frac{k-\tau_{\ell(k)}}{\sqrt{k}} \overset{(d)}{\rightarrow} 0 \quad \text{ and } \quad  \frac{\tau_{\ell(k)+1}-\tau_{\ell(k)}}{\sqrt{k}} \overset{(d)}{\rightarrow} 0 \quad \text{ for } k\rightarrow \infty . 
\end{equation} 
Hence, we get from \eqref{eq:NormalConvergence2} and \eqref{eq:NormalConvergence3} that the first and third terms on the right-hand side of \eqref{eq:SplitCLT} converge to $0$ in probability, and by using \eqref{eq:NormalConvergence} for the second term we conclude.
\end{proof}

\begin{remark}\label{rem:samediffusivity}
	\rm{
	Since the above proof works for all $\lambda\geq 0$, it follows that the diffusivity $\sigma^2$ for $\lambda=0$ is the same as in Theorem~\ref{thm:pssinvariance}. Moreover, note that we only prove an annealed invariance principle in Proposition \ref{pro:Invariance}, but conjecture that also a quenched invariance principle holds; see \cite{BZ:InvarianceQuenched} for sufficiently large values of $\mu>0$.
	}
\end{remark}

Recall from Proposition \ref{pro:RegenerationSpeed} that the speed $v(\lambda)$ is equal to $\E_{\lambda}[X^1_{\tau_1}]/\estart{\tau_1}{}$ and $\estart{\tau_1}{}$ does not depend on $\lambda$.  
For every $t\geq 0$ we let $R(t)$, respectively $L(t)$, be the number of steps to the right, respectively to the left, that $X^1$ performed by time $t$ and $\cU(t)$ be the
number of steps that were carried out by the walker until time $t$.
Let $R_a(t)$, respectively $L_a(t)$, be the number of attempted jumps in the $\eone$ direction, respectively in the $-\eone$ direction, of $X^1$ until time $t$ and recall that $\cU_a(t)$ is the total number of attempted jumps of $X$ until time $t$. Note that $\cU_a(t) -\cU(t)$ is the number of jumps that were attempted but not carried out. Writing $R = R(\tau_1), L= L(\tau_1)$ and 
$R_a = R_a(\tau_1), L_a = L_a(\tau_1)$ we will show the following.

\begin{lemma}\label{lem:uptotau} 
We have
\begin{equation}
\E_{\lambda}[X_{\tau_1}^1] = \E_{0}\left[(R-L ) e^{\lambda (R_a - L_a) }\left(\frac{2d}{Z_\lambda}\right)^{\cU_a(\tau_1)}\right]\, .
\end{equation}
\end{lemma}

\begin{proof}[\bf Proof]
Define the $\sigma$-field
${\cal F}_t$ generated by the evolution of the percolation, the exponential waiting times of the particle, the decisions of the particle up to time $t$ and the Poisson process $(N_t)_{t \geq 0}$ used in the definition of the regeneration times.  By saying "evolution of the percolation" we include the times of updates of $\eta = (\eta_s)_{s \geq 0}$ even if the state of the edge does not change.  
Consider the Radon-Nikodym derivative $\frac{d \mathbb{P}_\lambda}{d \mathbb{P}_0}\big|_{{\cal F}_{t}}$
of $\mathbb{P}_\lambda$ with respect to $\mathbb{P}_0$ on ${\cal F}_{t}$.
We then have
\begin{align}\begin{split}\label{prodRN1}
&M_t : = \left. \frac{d \mathbb{P}_\lambda}{d \mathbb{P}_0}\right|_{{\cal F}_{t}} 
= e^{\lambda (R_a(t) -L_a(t)) }
\left( \frac{2d}{Z_\lambda}\right)^{\cU_a(t)}.
\end{split}
\end{align}
Indeed, conditional on the evolution of the percolation, since it is the same for both measures, for the 
$\lambda$-biased walk, each attempted jump to the right (resp.\ to the left) has probability $e^\lambda/Z_\lambda$ (resp.\ $e^{-\lambda}/Z_\lambda$), while each other direction has probability $1/Z_\lambda$. For the unbiased walk the respective probabilities are all equal to $1/(2d)$. 
Note that $\tau_1$ is a stopping time with respect to the filtration $({\cal F}_t)_{t \geq 0}$, and that $(M_{t \wedge \tau_1})_{t \geq 0}$ is a martingale with respect to $({\cal F}_t)_{t \geq 0}$.
Since $\mathbb{P}_\lambda$ is absolutely continuous with respect to $\mathbb{P}_0$ on ${\cal F}_{\tau_1}$ (this follows simply from the fact that $\tau_1< \infty$ almost surely),
$(M_{t \wedge \tau_1})_{t \geq 0}$ is a uniformly integrable martingale.
Hence, 
\begin{align}\begin{split}\label{prodRNtau}
&M_{\tau_1} = \left. \frac{d \mathbb{P}_\lambda}{d \mathbb{P}_0}\right|_{{\cal F}_{\tau_1}} 
= e^{\lambda (R_a -L_a) }
\left( \frac{2d}{Z_\lambda}\right)^{\cU_a(\tau_1)}, 
\end{split}
\end{align}
i.e. $M_{\tau_1}$ is the Radon-Nikodym derivative of 
$\mathbb{P}_\lambda$ with respect to $\mathbb{P}_0$ on ${\cal F}_{\tau_1}$.
Hence
\begin{align}
\E_{\lambda}[X_{\tau_1}^1] = \E_{0}\left[X_{\tau_1}^1 e^{\lambda (R_a - L_a) }\left(\frac{2d}{Z_\lambda}\right)^{\cU_a(\tau_1)}\right] = \E_{0}\left[(R-L ) e^{\lambda (R_a - L_a) }\left(\frac{2d}{Z_\lambda}\right)^{\cU_a(\tau_1)}\right]\, ,
\end{align}
which gives the result.
\end{proof}
Let 
\begin{equation}\label{eq:ProbabilityUnbiased}
p_0(k_a, \ell_a, k, \ell, m) : = \P_{0}\left((R_a,L_a)=(k_a,\ell_a), (R,L)=(k,\ell), \cU_a(\tau_1) = m \right) , 
\end{equation}
  where $\P_{0}$ denotes the law of the symmetric simple random walk on dynamical percolation. Then 
\begin{align} \label{eq:SumRepresenationfAlpha}
f(\lambda) &:=\E_{\lambda}[X_{\tau_1}^1]  =\estart{X_{\tau_1}^1 
e^{\lambda (R_a-L_a) }\left( \frac{2d}{Z_\lambda}\right)^{\cU_a(\tau_1)}}{0} \nonumber \\
&=\sum_{m\in \N} \sum_{\substack{k_a+\ell_a \leq  m \\  k\leq k_a, \ell \leq \ell_a}} (k-\ell)e^{\lambda (k_a-\ell_a)} \left(\frac{2d}{Z_\lambda} \right)^m 
p_0(k_a, \ell_a, k, \ell, m) \, .
\end{align}

\begin{lemma}\label{lem:Differentiability} Let $\mu>0$ and $p\in (0,1)$. Then the speed $v(\lambda)$ is continuously differentiable in 
$\lambda>0$ and the derivative satisfies 
\begin{equation}\label{formula-deriv-speed}
v'(\lambda) = \frac{1}{\Ex{\tau_1}}\cdot \left( \estart{X_{\tau_1}^1 \left(R_a - L_a\right) }{\lambda}  - \frac{e^\lambda - e^{-\lambda}}{Z_\lambda}\cdot 
\estart{X_{\tau_1}^1\cdot  \cU_a(\tau_1)}{\lambda} \right).
\end{equation}
\end{lemma}

\begin{proof}[\bf Proof] 
We first prove that for all $\lambda>0$,
\begin{align}
\begin{split}\label{eq:derivativeoff}
\lim_{\delta\to 0} &\frac{f(\lambda+\delta) - f(\lambda)}{\delta}	\\&= \sum_{m\in \N} \sum_{\substack{k_a+\ell_a \leq  m \\  k\leq k_a, \ell \leq \ell_a}} (k-\ell)e^{\lambda (k_a-\ell_a)} \left(\frac{2d}{Z_\lambda} \right)^m 
\left( k_a-\ell_a -m\cdot \frac{Z_\lambda'}{Z_\lambda} \right)p_0(k_a, \ell_a, k, \ell, m),
\end{split}
\end{align}
where $Z_\lambda':=e^\lambda - e^{-\lambda}$. 
The last term equals
$$
\estart{X_{\tau_1}^1 (R_a - L_a) }{\lambda}  - \frac{e^\lambda - e^{-\lambda}}{Z_\lambda}\cdot 
\estart{X_{\tau_1}^1\cdot  \cU_a(\tau_1)}{\lambda} \, .
$$

Note that the sum appearing above divided by $\estart{\tau_1}{}$ is equal to the expression for the derivative given in the statement of the lemma.

A direct calculation shows that
\begin{align*}
\frac{f(\lambda+\delta) - f(\lambda)}{\delta}= \sum_{m\in \N} \sum_{\substack{k_a+\ell_a \leq  m \\  k\leq k_a, \ell \leq \ell_a}} (k-\ell)e^{\lambda (k_a-\ell_a)} \left(\frac{2d}{Z_\lambda} \right)^m  g(\delta)\cdot p_0(k_a, \ell_a, k, \ell, m) ,
\end{align*}
where the function $g$ is defined via 
\[
g(\delta) = \frac{e^{\delta(k_a-\ell_a)}Z_\lambda^m - Z_{\lambda+\delta}^m}{\delta Z_{\lambda+\delta}^m}.
\]
There exists a positive constant $c=c_d$ so that for all $\lambda$ and $\delta$ we have 
\begin{align}\label{eq:eqforratioofzlambda}
1\geq \frac{Z_\lambda}{Z_{\lambda+\delta}}\geq 1- c\delta. 
\end{align}
By taking $\delta<1/c$, and considering whether $\delta<1/m$ or $\delta\geq 1/m$, we see that there is a positive constant $C=C_d$ so that
\[
|g(\delta)|\leq Cm.
\]
Indeed, the upper bound on $g(\delta)$ follows by a Taylor expansion, while the lower bound follows using~\eqref{eq:eqforratioofzlambda}.
Therefore, we obtain that  uniformly for $\delta <1/c$,
\begin{align*}
\left|\frac{f(\lambda+\delta) - f(\lambda)}{\delta} \right| &\leq \sum_{m\in \N} \sum_{\substack{k_a+\ell_a \leq  m \\  k\leq k_a, \ell \leq \ell_a}} Cm^2 e^{\lambda (k_a-\ell_a)} \left(\frac{2d}{Z_\lambda} \right)^m 
p_0(k_a, \ell_a, k, \ell, m)
\\
&=C\cdot \estart{(\cU_a(\tau_1))^2}{\lambda}<\infty,
\end{align*}
where for the last bound we used Lemma~\ref{cor:exptailsofutauone}, since 
the distribution of $\cU_a(\tau_1)$ does not depend on~$\lambda$. We can thus apply  the dominated convergence theorem which allows us to  differentiate the summands in \eqref{eq:SumRepresenationfAlpha} with respect to $\lambda$ to get~\eqref{eq:derivativeoff}. Applying the dominated convergence theorem again we see that all the terms appearing in the expression for $v'(\lambda)$ are continuous functions in $\lambda$, and this finishes the proof of \eqref{formula-deriv-speed}.  \end{proof}

\begin{proposition}\label{einstein-x}
We have
\begin{equation*}
\lim_{\lambda\to 0} v'(\lambda)= \sigma^2,
\end{equation*}
where $\sigma^2 = \estart{(X_{\tau_1}^1)^2}{0}/\Ex{\tau_1} $ is the variance from Theorem~\ref{thm:pssinvariance}, see \eqref{einstein-a}.
\end{proposition} 

\begin{proof}
We see, taking $\lambda \to 0$ in \eqref{formula-deriv-speed} the statement follows
 if we show that
\begin{equation}\label{cov-is-var}
\estart{X_{\tau_1}^1 \left(R_a - L_a\right) }{0} = \estart{X_{\tau_1}^2}{0}\, .
\end{equation}
Turning to the proof of \eqref{cov-is-var}, we first
note that \eqref{cov-is-var} is equivalent to
\begin{equation}\label{cov-is-var2}
\estart{(R-L) \left(R_{\rm supp} - L_{\rm supp}\right) }{0} = 0, 
\end{equation}
where we write $R_{\rm supp} = R_a - R$ and $L_{\rm supp} = L_a - L$ for the jumps to the right or left respectively that were not carried out (the subscript stands for "suppressed jumps").
Similarly as in the proof of Proposition \ref{pro:RegenerationSpeed}, one can show that, with $R_{\rm supp}(t) = R_a(t) - R(t)$, $L_{\rm supp}(t) = L_a(t) - L(t)$,
\begin{equation}\label{cov-is-var2}
\lim\limits_{t \to \infty}\frac{1}{t} \estart{(R(t)-L(t)) \left(R_{\rm supp}(t) - L_{\rm supp}(t)\right) }{0}
= \frac{1}{\E[\tau_1]} \estart{(R-L) \left(R_{\rm supp} - L_{\rm supp}\right) }{0} \, .
\end{equation}
Hence, it suffices to prove that for all $t > 0$,
\begin{equation}\label{orthogonality}
 \estart{(R(t)-L(t)) \left(R_{\rm supp}(t) - L_{\rm supp}(t)\right) }{0}= 0\, .
\end{equation}
We will provide a measure-preserving transformation such that $R(t)-L(t)$ is 
an antisymmetric function but $R_{\rm supp}(t) - L_{\rm supp}(t)$ is a symmetric function under this transformation, implying that 
\eqref{orthogonality} holds true.
Indeed, the law of the environment and the walk under $\mathbb{P}_0$  is invariant under time-reversal. More precisely, fix $t> 0$ and
let $\tilde \eta_s = \eta_{t-s}$, $0 \leq s \leq t$ and $\tilde X_s = X_{t-s}- X_t$, $0 \leq s \leq t$. 
Consider all the times and decisions of the particle in the path $(X_s)_{0 \leq s \leq t}$ concerning jumps to the right or to the left that were suppressed. Call these times
$r_1, r_2, \ldots ,r_{R_{\rm supp}(t)}$ (for suppressed jumps to the right) and $\ell_1, \ell_2, \ldots ,\ell_{L_{\rm supp}(t)}$ (for suppressed jumps to the left) respectively.
Note that 
$$
\left((\eta_s)_{0 \leq s \leq t}, (X_s)_{0 \leq s \leq t},(r_1, r_2, \ldots ,r_{R_{\rm supp}(t)}), (\ell_1, \ell_2, \ldots ,\ell_{L_{\rm supp}(t)})\right) 
$$ 
has the same law under $\mathbb{P}_0$ as 
$$
\left( (\tilde \eta_s)_{0 \leq s \leq t}, (\tilde X_s)_{0 \leq s \leq t},(t- r_1, \ldots ,t- r_{R_{\rm supp}(t)}), (t-\ell_1,  \ldots ,t-\ell_{ L_{\rm supp}(t)} )\right)\, .
$$
Since
$R(t)-L(t) = X_t = - \tilde X_t$ and $R_{\rm supp}(t)$ and $L_{\rm supp}(t)$ are the same for both processes,
 see Figure \ref{fig:TimeRev}, we conclude that \eqref{orthogonality} holds true.  \end{proof} 
\begin{figure}\centering \begin{tikzpicture}[scale=0.8]
\label{fig:TimeRev}

\foreach \x in{2,...,5}{
	\draw[gray!50,thin](1.6,\x) to (5.4,\x);  
	\draw[gray!50,thin](\x,1.6) to (\x,5.4);   
}

	\draw[red,line width=1.5pt] (3,5) --(3,4);
	\draw[red,line width=1.5pt] (3,4) --(4,4);
	\draw[red,line width=1.5pt] (4,4) --(4,3);	
	\draw[red,line width=1.5pt] (4,3) --(3,3);	
	\draw[red,line width=1.5pt] (3,3) --(3,2);	
	\draw[red,line width=1.5pt] (3,2) --(2,2);	
	\draw[red,line width=1.5pt] (2,2) --(2,3);	
	
	\draw[red,line width=1.5pt,densely dotted] (2,4) --(3,4);
	\draw[red,line width=1.5pt,densely dotted] (4,4) --(5,4);
	\draw[red,line width=1.5pt,densely dotted] (4,3) --(5,3);

\draw[fill,black](3-0.1,5-0.1) rectangle (3+0.1,5+0.1);

	\draw[line width=1.pt] (3+0.2,4.4+0.2) --(3,4.4)--(3-0.2,4.4+0.2);
	\draw[line width=1.pt] (4+0.2,3.4+0.2) --(4,3.4)--(4-0.2,3.4+0.2);
	\draw[line width=1.pt] (3+0.2,2.4+0.2) --(3,2.4)--(3-0.2,2.4+0.2);
	\draw[line width=1.pt] (2-0.2,2.6-0.2) --(2,2.6)--(2+0.2,2.6-0.2);

	\draw[line width=1.pt] (2.4+0.2,4+0.2) --(2.4,4)--(2.4+0.2,4-0.2);		
	\draw[line width=1.pt] (3.6-0.2,4-0.2) --(3.6,4)--(3.6-0.2,4+0.2);	
	\draw[line width=1.pt] (4.6-0.2,4-0.2) --(4.6,4)--(4.6-0.2,4+0.2);		

	\draw[line width=1.pt] (3.4+0.2,3+0.2) --(3.4,3)--(3.4+0.2,3-0.2);		
	\draw[line width=1.pt] (4.6-0.2,3-0.2) --(4.6,3)--(4.6-0.2,3+0.2);		

	\draw[line width=1.pt] (2.4+0.2,2+0.2) --(2.4,2)--(2.4+0.2,2-0.2);

	\node[scale=0.7] (ab) at (3.3,4.7) {$1$};	
	\node[scale=0.7] (ab) at (2.5,4.35) {$2$};		
	\node[scale=0.7] (ab) at (3.5,4.35) {$3$};		
	\node[scale=0.7] (ab) at (4.5,4.35) {$4$};		
	\node[scale=0.7] (ab) at (4.3,3.7) {$5$};
	\node[scale=0.7] (ab) at (4.5,3.35) {$6$};		
	\node[scale=0.7] (ab) at (3.5,3.35) {$7$};	
	\node[scale=0.7] (ab) at (3.3,2.6) {$8$};
	\node[scale=0.7] (ab) at (2.5,2.35) {$9$};
	\node[scale=0.7] (ab) at (2.3,2.7) {$10$};

	\def \y {6};
	
\foreach \x in{2,...,5}{
	\draw[gray!50,thin](1.6+\y,\x) to (5.4+\y,\x);  
	\draw[gray!50,thin](\x+\y,1.6) to (\x+\y,5.4);   
}

	\draw[red,line width=1.5pt] (3+\y,5) --(3+\y,4);
	\draw[red,line width=1.5pt] (3+\y,4) --(4+\y,4);
	\draw[red,line width=1.5pt] (4+\y,4) --(4+\y,3);	
	\draw[red,line width=1.5pt] (4+\y,3) --(3+\y,3);	
	\draw[red,line width=1.5pt] (3+\y,3) --(3+\y,2);	
	\draw[red,line width=1.5pt] (3+\y,2) --(2+\y,2);	
	\draw[red,line width=1.5pt] (2+\y,2) --(2+\y,3);	
	
	\draw[red,line width=1.5pt,densely dotted] (2+\y,4) --(3+\y,4);
	\draw[red,line width=1.5pt,densely dotted] (4+\y,4) --(5+\y,4);
	\draw[red,line width=1.5pt,densely dotted] (4+\y,3) --(5+\y,3);

\draw[fill,black](2-0.1+\y,3-0.1) rectangle (2+0.1+\y,3+0.1);

		\draw[line width=1.pt] (\y+3-0.2,4.6-0.2) --(\y+3,4.6)--(\y+3+0.2,4.6-0.2);
	\draw[line width=1.pt] (\y+4-0.2,3.6-0.2) --(\y+4,3.6)--(\y+4+0.2,3.6-0.2);
	\draw[line width=1.pt] (\y+3-0.2,2.6-0.2) --(\y+3,2.6)--(\y+3+0.2,2.6-0.2);
	\draw[line width=1.pt] (\y+2+0.2,2.4+0.2) --(\y+2,2.4)--(\y+2-0.2,2.4+0.2);

	\draw[line width=1.pt] (\y+2.4+0.2,4+0.2) --(\y+2.4,4)--(\y+2.4+0.2,4-0.2);		
	\draw[line width=1.pt] (\y+3.4+0.2,4+0.2) --(\y+3.4,4)--(\y+3.4+0.2,4-0.2);	
	\draw[line width=1.pt] (\y+4.6-0.2,4-0.2) --(\y+4.6,4)--(\y+4.6-0.2,4+0.2);		

	\draw[line width=1.pt] (\y+3.6-0.2,3-0.2) --(\y+3.6,3)--(\y+3.6-0.2,3+0.2);		
	\draw[line width=1.pt] (\y+4.6-0.2,3-0.2) --(\y+4.6,3)--(\y+4.6-0.2,3+0.2);		

	\draw[line width=1.pt] (\y+2.6-0.2,2-0.2) --(\y+2.6,2)--(\y+2.6-0.2,2+0.2);	
	
		\node[scale=0.7] (ab) at (\y+3.3,4.7) {$10$};	
	\node[scale=0.7] (ab) at (\y+2.5,4.35) {$9$};		
	\node[scale=0.7] (ab) at (\y+3.5,4.35) {$8$};		
	\node[scale=0.7] (ab) at (\y+4.5,4.35) {$7$};		
	\node[scale=0.7] (ab) at (\y+4.3,3.7) {$6$};
	\node[scale=0.7] (ab) at (\y+4.5,3.35) {$5$};		
	\node[scale=0.7] (ab) at (\y+3.5,3.35) {$4$};	
	\node[scale=0.7] (ab) at (\y+3.3,2.6) {$3$};
	\node[scale=0.7] (ab) at (\y+2.5,2.35) {$2$};
	\node[scale=0.7] (ab) at (\y+2.3,2.7) {$1$};
	
\end{tikzpicture}
\caption{Visualization of the trajectory of a random walk in dynamical percolation and its time-reversal. The black dot marks the origin $(0,0)$, while the numbers refer to the order in which jumps are attempted (and performed) in the two processes. A line is dotted when a jump was attempted, but suppressed. }
\end{figure}
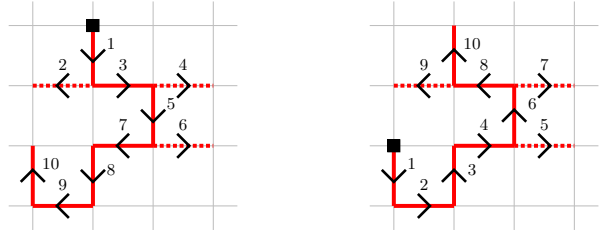

\begin{proposition}\label{pro:Einstein} Fix $p\in(0,1)$ and $\mu>0$. Then the speed function $\lambda \mapsto v(\lambda)$ is strictly positive for all $\lambda>0$. 
\end{proposition}

\begin{proof}
We have, using again the time-reversal argument in the proof of Proposition \ref{einstein-x} for the third equality
\begin{align}\label{first-for-t} 
\begin{split}
&\quad  \estart{X_t^1}{\lambda} = \E_{0}\left[\left(R(t) - L(t)\right)e^{\lambda (R_a(t) -L_a(t)) }\left(\frac{2d}{Z_\lambda}\right)^{\cU_a(t)}\right] \\
&= \E_{0}\left[\left(R(t) - L(t)\right)e^{\lambda (R(t)+ R_{\rm supp}(t) - L(t)- L_{\rm supp}(t)) }\left(\frac{2d}{Z_\lambda}\right)^{\cU_a(t)}\right] \\
&= \E_{0}\left[\left(L(t) - R(t)\right)e^{\lambda (L(t)+ R_{\rm supp}(t) - R(t)- L_{\rm supp}(t)) }\left(\frac{2d}{Z_\lambda}\right)^{\cU_a(t)}\right] \\
&= \frac{1}{2} \E_{0}\left[\left(R(t) - L(t)\right)e^{\lambda (R_{\rm supp}(t) - L_{\rm supp}(t)) }\left(e^{\lambda (R(t) -L(t)) }- e^{-\lambda (R(t) -L(t)) }\right)\left(\frac{2d}{Z_\lambda}\right)^{\cU_a(t)}\right] \\
&> 0,
\end{split}
\end{align}
where the last inequality follows since the function $x \mapsto x\left(e^{\lambda x} - e^{-\lambda x}\right)$ is non-negative and strictly positive for $x\neq 0$ and we have $\mathbb{P}_0(R(t)\neq L(t)) > 0$.
(The fourth equality comes from adding the two previous lines and dividing by $2$).
Hence, we know that for fixed $t$, $\estart{X_t^1}{\lambda} > 0$ and we have to show that $\estart{X_{\tau_1}^1}{\lambda} > 0$. 
Going back to \eqref{first-for-t},
\begin{align*}
&\quad  \estart{X_t^1}{\lambda} \\
&= \frac{1}{2} \E_{0}\left[\left(R(t) - L(t)\right)e^{\lambda (R_{\rm supp}(t) - L_{\rm supp}(t)) }
\left(e^{\lambda (R(t) -L(t)) }- e^{-\lambda (R(t) -L(t)) }\right)\left(\frac{2d}{Z_\lambda}\right)^{\cU_a(t)}\right] \\
&= \frac{1}{2} \E_{\lambda}\left[\left(R(t) - L(t)\right) e^{\lambda (R_{\rm supp}(t) - L_{\rm supp}(t)) }\left(e^{\lambda (R(t) -L(t)) }- 
e^{-\lambda (R(t) -L(t)) }\right)e^{-\lambda (R_a(t) - L_a(t)) }\right] \, ,
\end{align*}
where we used \eqref{prodRN1} for the last equality. Hence
\begin{align*}
&\quad  \estart{X_t^1}{\lambda} \\
&= \frac{1}{2} \E_{\lambda}\left[\left(R(t) - L(t)\right)e^{-\lambda (R(t) - L(t)) }\left(e^{\lambda (R(t) -L(t)) }- e^{-\lambda (R(t) -L(t)) }\right) \right] \\
&= \frac{1}{2}\E_{\lambda}\left[\left(R(t) - L(t)\right)\left( 1- e^{-2\lambda (R(t) - L(t)) }\right) \right] \, .
\end{align*}
This implies
\begin{equation}\label{t-is-fixed}
\estart{X_t^1}{\lambda} = - \E_{\lambda}\left[\left(R(t) - L(t)\right) e^{-2\lambda (R(t) - L(t)) } \right] \, .
\end{equation}
We will show that, in the same spirit as in \eqref{eq:SpeedViaTau2},
\begin{equation}\label{from-t-to-tau}
\lim\limits_{ t \to \infty}\frac{1}{t}\E_{\lambda}\left[\left(R(t) - L(t)\right) e^{-2\lambda (R(t) - L(t)) } \right]  
= \frac{1}{\E[\tau_1]} \E_{\lambda}\left[\left(R - L\right) e^{-2\lambda (R - L) } \right]
\end{equation}
Assuming \eqref{from-t-to-tau}, dividing by $t$ in \eqref{t-is-fixed} and letting $t \to \infty$ gives
\begin{equation*}
\estart{X_{\tau_1}^1}{\lambda} = -\E_{\lambda}\left[\left(R - L\right) e^{-2\lambda (R - L) }\right] = \frac{1}{2}\E_{\lambda}\left[\left(R - L\right) \left(1- e^{-2\lambda (R - L) } \right)\right]  > 0,
\end{equation*}
since the function $x \mapsto x\left(1- e^{-2\lambda x}\right)$ is non-negative and strictly positive for $x \neq 0$ and we have 
$\mathbb{P}_\lambda(R\neq L) > 0$.
It remains to prove \eqref{from-t-to-tau}.
We claim that for all $t>s$ we have
\begin{equation}\label{expect-IS-1}
\escond{e^{-2\lambda (R(t) - L(t)) }}{{\cal F}_s}{\lambda}=e^{-2\lambda (R(s) - L(s)) }\, 
\end{equation}
and defer the proof.
Hence, 
$e^{-2\lambda (R(t) - L(t)) }$ can act as a Radon-Nikodym derivative. More precisely, define the probability measure 
$\overline{\mathbb{P}}_\lambda$ by
\begin{equation*}
\left. \frac{d \overline{\mathbb{P}}_\lambda}{d \mathbb{P}_\lambda}\right|_{{\cal F}_t} 
= e^{-2\lambda (R(t) -L(t)) }\, .
\end{equation*}
Note that, as in the proof of \eqref{prodRNtau}, for every $i$, the regeneration time $\tau_i$ is a stopping time with respect to the filtration $({\cal F}_t)_{t \geq 0}$, and 
$\overline{\mathbb{P}}_\lambda$ is absolutely continuous with respect to $\mathbb{P}_\lambda$ on ${\cal F}_{\tau_i}$ (recall the beginning of the proof of Lemma \ref{lem:uptotau}). Hence
\begin{align}\label{eq:explicitchangeofmeasure}
\left. \frac{d \overline{\mathbb{P}}_\lambda}{d \mathbb{P}_\lambda}\right|_{{\cal F}_{\tau_i}} 
= e^{-2\lambda (R(\tau_i) -L(\tau_i)) }\, .
\end{align}
Therefore, taking $i=1$ we get 
\begin{equation}\label{change-prob}
\E_{\lambda}\left[\left(R - L\right) e^{-2\lambda (R - L) } \right] = \overline{\E}_{\lambda}\left[R - L \right]\, ,
\end{equation}
where $\overline{\E}_{\lambda}$ is the expectation with respect to $\overline{\P}_{\lambda}$.
Using the fact that the $(\tau_i)$'s are regeneration times under $\mathbb{P}_\lambda$ and~\eqref{eq:explicitchangeofmeasure} immediately yield that the $(\tau_i)$'s are also regeneration times under~$\overline{\mathbb{P}}_\lambda$.
We can now proceed as in the proof of \eqref{eq:SpeedViaTau2}, replacing the probability measure $\mathbb{P}_\lambda$ by~$\overline{\P}_{\lambda}$, to get
\begin{align*}
\lim\limits_{ t \to \infty}\frac{1}{t}\E_{\lambda}\left[\left(R(t) - L(t)\right) e^{-2\lambda (R(t) - L(t)) } \right]  
&= \lim\limits_{ t \to \infty}\frac{1}{t}\overline\E_{\lambda}\left[R(t) - L(t)\right] \\
&=  \frac{1}{\E[\tau_1]}\overline\E_{\lambda}\left[R - L\right] \\
&=  \frac{1}{\E[\tau_1]} \E_{\lambda}\left[\left(R - L\right) e^{-2\lambda (R - L) } \right] \, ,  \\
\end{align*}
where we used \eqref{change-prob} for the last equality.
Finally, we prove \eqref{expect-IS-1}.
Note that
it suffices to establish the case $s=0$ and $(X_0,\eta_0)=(x,\zeta)$. Omitting the starting state from the notation we have 
\begin{align*}
\E_{\lambda}\left[e^{-2\lambda (R(t) - L(t)) }\right]
&= \E_{0}\left[e^{-2\lambda (R(t) - L(t)) } e^{\lambda (R_a(t)- L_a(t)) }\left(\frac{2d}{Z_\lambda}\right)^{\cU_a(t)}\right] \\
&= \E_{0}\left[e^{-\lambda (R(t) - L(t)) } e^{\lambda (R_{\rm supp}(t) - L_{\rm supp}(t)) }\left(\frac{2d}{Z_\lambda}\right)^{\cU_a(t)}\right] \\
& = \E_{0}\left[e^{\lambda (R(t) - L(t)) } e^{\lambda (R_{\rm supp}(t) - L_{\rm supp}(t)) }\left(\frac{2d}{Z_\lambda}\right)^{\cU_a(t)}\right] \\
&= \E_{0}\left[e^{\lambda (R_a(t)- L_a(t)) }\left(\frac{2d}{Z_\lambda}\right)^{\cU_a(t)}\right]  = 1\,  ,\\
\end{align*}
where we again used the time-reversal argument in the proof of Proposition \ref{einstein-x} for the second-to-last equality.
\end{proof}

\begin{proof}[\bf Proof of Theorem \ref{thm:main}] Theorem \ref{thm:main} is now an immediate consequence of Proposition~\ref{pro:RegenerationSpeed}, Proposition~\ref{einstein-x}, Proposition~\ref{pro:Einstein} and  Lemma~\ref{lem:Differentiability}.
\end{proof}

 \section{Monotonicity of the speed}\label{sec:monotonicity}

 In this section we prove Theorem~\ref{thm:mainTwoDim}. In order to do so, we first establish an asymptotic expression for the speed that is valid for large values of the bias $\lambda$. 
We recall the definition from~\eqref{eq:defz} of $Z_\lambda=e^\lambda+e^{-\lambda}+2d-2$.

  \begin{proposition}\label{pro:MonotoneSubsequence} For $d \geq 1$, let $(X,\eta)$ be a $\lambda$-biased random walk on dynamical percolation on $\Z^d$ with parameters $\mu>0$ and $p\in (0,1)$. There exists some $\lambda_0=\lambda_0(\mu,d)$ such that for all $\lambda>\lambda_0$, 
\begin{equation*}
{v}(\lambda) = \frac{\mu p}{1-p+\mu} -  \frac{(2d-2) p}{(1-p+\mu)^2}(\mu^2-p(1-p)) Z_{\lambda}^{-1}  + \mathcal{O}(e^{-2\lambda})  ,
\end{equation*}
where the implicit constant in $\cO$ depends on $\mu$ and $d$.
\end{proposition}

\begin{remark}
\rm{
 Note that the speed $v(\lambda)$ converges to $\mu p (1-p+\mu)^{-1}$ as $\lambda \rightarrow \infty$ in agreement with the $1$-dimensional case as we will see in  Proposition~\ref{thm:mainOneDim}.
}	
\end{remark}

The above proposition proves the monotonicity of $v(\lambda)$ along arithmetic progressions for large $\lambda$. In order to prove Theorem~\ref{thm:mainTwoDim} we also need to obtain a control on the derivative of  $v(\lambda)$ that is valid for large values of $\lambda$.

\begin{lemma}\label{lem:ApproximateDerivate} Let $d \geq 1$.  Then for all $\mu>0$ and $p\in (0,1)$ there exists $\lambda_0=\lambda_0(\mu)$  and  constants~$c_{\mu}>0 $ and $ C_{\mu,p} \in \R$ so that for all $\lambda\geq \lambda_0$ we have 
\begin{equation}\label{eq:ApproximateDerivate}
\left| v^{\prime}(\lambda) - C_{\mu,p} \exp(-\lambda)  \right| \leq c_{\mu}\exp(-2\lambda) \, .
\end{equation} 
\end{lemma}

We now have all the tools needed in order to conclude the proof of Theorem \ref{thm:mainTwoDim}. We defer the proofs of Proposition~\ref{pro:MonotoneSubsequence} and Lemma~\ref{lem:ApproximateDerivate} to Sections~\ref{sec:asymptoticexpression} and~\ref{sec:LargeBiasLimit} respectively.

\begin{proof}[\bf Proof of Theorem \ref{thm:mainTwoDim}]
Using Lemma \ref{lem:ApproximateDerivate}, it suffices to study the constant $C_{\mu,p}$ in \eqref{eq:ApproximateDerivate} and show that $C_{\mu,p}>0$ when $\mu^2 > p(1-p)$ as well as $C_{\mu,p}<0$ when $\mu^2 < p(1-p)$. Since we know that the speed is continuously differentiable by Theorem \ref{thm:main}, we get that for all $s>0$ large enough
\begin{equation}
v(2s)-v(s)=\int_{s}^{2s} v^{\prime}(t) \dif t =  C_{\mu,p} \exp(-s) + \mathcal{O}(\exp(-2s) ) \, .
\end{equation} Taking now $s=\lambda$ sufficiently large, we get from Proposition~\ref{pro:MonotoneSubsequence} that
\begin{equation}
C_{\mu,p}=  \frac{(2d-2) p}{(1-p+\mu)^2}(\mu^2-p(1-p)) \, ,
\end{equation} allowing us to conclude, since we get 
\[
v'(\lambda) = C_{\mu,p} e^{-\lambda}+\cO(e^{-2\lambda}), 
\]
and hence the sign of $v'(\lambda)$ agrees with the sign of $C_{\mu,p}$ for all $\lambda$ sufficiently large.
\end{proof}

\subsection{Speed for $d=1$}\label{sec:SpeedRWOneDim}

In this section we focus on dimension $1$, where one can use the obvious coupling between two random walks with different biases to obtain that the speed is increasing as a function of the bias. We investigate the limiting speed and the rate of convergence to the limit for large $\lambda$.

In the following proposition we establish strict monotonicity as well as an explicit form for the speed in the totally asymmetric biased random walk case, where the random walk only attempts jumps to the right.

\begin{lemma}\label{lem:speedinonedim}
	Let $(X,\eta)$ be a totally asymmetric biased random walk in dynamical percolation on $\Z$ with parameters $\mu$ and $p$ that attempts a  jump to the right at rate $1$. Then the speed $\overline{v}$ satisfies
	\[
	\overline{v}= \frac{\mu p}{1-p+\mu}.
	\]
\end{lemma}

\begin{proof}[\bf Proof]
Suppose $X_0=0$ and $\eta_0\sim \pi$. Let $S$ be the first time that $X$ jumps along the edge $e=\{0,1\}$. Then $\overline{v} = \Ex{S}^{-1}$.  In order to compute $\E[S]$, we use the following recursion,
\begin{align*}
\Ex{S} = p+ (1-p) \Big( 1+ \frac 1 \mu + \Ex{S}  \Big).
\end{align*}
Indeed,  at the time of the first attempted jump of the walk the edge $e$ is open with probability $p$.  Otherwise,  the walk has to wait an exponential time with parameter $\mu$  until $e$ refreshes,  and when this happens the situation is the same as at time $0$. This recursion readily implies
\begin{align*}
\E[S]= \frac{1-p+\mu}{\mu p} ,
\end{align*}
which gives the result.
\end{proof}

\begin{proposition}[Monotonicity and asymptotic speed for $d=1$]\label{thm:mainOneDim} Let $(X_t,\eta_t)_{t\geq 0}$ be a biased random walk in dynamical percolation on $\Z$ with parameters $\mu$ and $p$.
Then the speed function $v(\lambda)$ from \eqref{eq:SpeedBRW} is strictly increasing for all $\lambda>0$ and satisfies
\begin{equation}
\lim_{\lambda \rightarrow \infty} v_{\mu,p}(\lambda) =  \frac{\mu p}{1-p+\mu} 
\end{equation} for all choices of $p\in (0,1)$ and $\mu>0$.
\end{proposition}

\begin{proof}[\bf Proof] 

We start by arguing that the speed is strictly increasing in $\lambda>0$.
We construct a coupling $\mathbf{P}$ between a $\lambda_1$-biased random walk $(X_t,\eta_t)_{t\geq 0}$ and a $\lambda_2$-biased random walk $(\til{X}_t, \til{\eta}_t)_{t\geq 0}$ on dynamical percolation on $\Z$ with $0<\lambda_1<\lambda_2$. We take the same environment for both walks and we let them attempt jumps in the following way: whenever the two random walks are at the same location, we couple them by using the same exponential $1$ clocks to determine the jump times and then moving them both to the right with probability $e^{\lambda_1}/(e^{\lambda_1}+e^{-\lambda_1})$, moving $\til{X}$ to the right and $X$ to the left with probability $e^{\lambda_2}/(e^{\lambda_2}+e^{-\lambda_2}) - e^{\lambda_1}/(e^{\lambda_1}+e^{-\lambda_1})$ and moving them both to the left otherwise. If the two walks are in different locations, we let them attempt jumps in the common environment using independent exponential $1$ clocks. 

Recall the construction of the infected set from Section \ref{sec.RegenerationTimes} and the definition of copies of edges. 
We define the following \textbf{modified infected set} $(\overline{I}_t)_{t \geq 0}$, where for every $t\geq 0$, $\overline{I}_t$ is a set containing copies of edges. Suppose that for some $t \geq 0$, both random walks are at the same position. If the two random walks examine the same edge $e_i$ for some $i\in \N$ and no copy of $e_i$ is contained in $\overline{I}_{t_-}$, we set
\begin{equation*}
\overline{I}_{t} := \overline{I}_{t_-} \cup \{e_{i,1} \} .
\end{equation*} 
Otherwise, we add to $\overline{I}_t$ the copy $e_{i,j}$ of $e_i$ with the smallest index $j$ such that $e_{i,j} \notin \overline{I}_{t_-}$. If the random walks are at the same position, but examine different edges, we add for both the copies of the edges with the smallest index as above to the modified infected set. When the two random walks are at different positions, recall that according to the coupling, the two random walks  perform jumps according to independent exponential $1$ clocks. Whenever an edge is examined by one of the two random walks, we add  its respective copy to the modified infected set.

Let $(\overline{N}_t)_{t \geq 0}$ be a Poisson process with time dependent intensity $\mu |\overline{I}_t|$. Whenever a clock of this process rings at time $t$, we choose an index uniformly at random from $\{1,\dots,|\overline{I}_t|\}$ and remove the copy of the edge with this index in $\overline{I}_t$ according to the ordering $\preceq$ of edges from Section~\ref{sec.RegenerationTimes}. If the picked copy is of the form $e_{i,1}$ for some $i\in \N$, we also refresh the state of the edge $e_i$ in the common environment $\eta_t$ for the two walkers, i.e.\ we set $\eta_t(e_i)=1$ with probability $p$, and $\eta_t(e_i)=0$, otherwise. For all edges $e_j$ with $e_{j,1} \notin \overline{I}_t$, we use independent rate $\mu$ Poisson clocks to determine when the respective edge is updated in the environment for the two random walks. 

Note that under this coupling $\mathbf{P}$ the biased random walks on dynamical percolation $(X_t,\eta_t)_{t\geq 0}$ and~$(\til{X}_t,\til{\eta}_t)_{t\geq 0}$ have marginally the correct law. We define
\begin{equation}
\overline{\tau} := \inf\{ t > 0 \colon \overline{I}_{t} = \emptyset \text{ and } \overline{I}_{t^{\prime}} \neq \emptyset \text{ for some } t^{\prime} \in (0,t)  \}  .
\end{equation} 
Since the process $(|\overline{I}_t|)_{t \geq 0}$ is dominated from above by a biased random walk on $\{0,1,\dots\}$ with transition rates $q(i,i-1)= \mu i $ and $q(i-1,i+1)=2$ for all $i\in \N$, a similar argument as in Lemma~\ref{lem:ExpoTailsTau} (see also the proof of Lemma~\ref{cor:exptailsofutauone}) shows that that the random variable~$\overline{\tau}$ has all finite moments. Applying now the same arguments as in the proof of Proposition \ref{pro:RegenerationSpeed}, we get that
\begin{equation}
v(\lambda_1) = \frac{\mathbf{E}[X_{\overline{\tau}}]}{\mathbf{E}[\overline{\tau}]} \quad \text{ and } \quad v(\lambda_2) = \frac{\mathbf{E}[\til{X}_{\overline{\tau}}]}{\mathbf{E}[\overline{\tau}]}, 
\end{equation} 
where we write $\mathbf{E}$ for the expectation with respect to $\mathbf{P}$. Note that $\mathbf{P}(X_{\overline{\tau}} \leq \til{X}_{\overline{\tau}})=1$. Considering the event that both random walks jump into different directions, and the respective edge copies get removed from the modified infected set before another jump occurs, we also get 
\[
\mathbf{P}(X_{\overline{\tau}}< \til{X}_{\overline{\tau}}) \geq 
\left(\frac{e^{\lambda_2}}{e^{\lambda_2}+ e^{-\lambda_2}} - \frac{e^{\lambda_1}}{e^{\lambda_1}+ e^{-\lambda_1}} \right)\cdot \left(\frac{\mu}{\mu+1} \right)^2>0
\] as $\lambda_1<\lambda_2$. This immediately implies that $v(\lambda_1)<v(\lambda_2)$, hence establishing strict monotonicity. 

Next, we investigate the speed when $\lambda \rightarrow \infty$. Let $X$ be a $\lambda$-biased walk and let~$\til{X}$ be a totally biased random walk as in Lemma~\ref{lem:speedinonedim}, i.e.\ it only attempts jumps to the right at rate $1$. We couple $X$ and $\til{X}$ by letting them both jump at the points of a Poisson process $\cP$ of rate $1$. At each jump attempt, we toss an independent coin and with probability $e^{\lambda}/(e^\lambda + e^{-\lambda})$ they both attempt a jump to the right, while with the complementary probability $X$ attempts a jump to the left while $\til{X}$ attempts a jump to the right. After the first time that they move in different directions, we couple their environments by using the same rate $\mu$ Poisson process for the removal of copies of edges of their respective infected sets. In this way, both infected sets always have the same size and the two walks have the same regeneration times. Letting $\tau_1$ be their first regeneration time, we then have 
\begin{align*}
	|v(\lambda)-\overline{v}| = \frac{1}{\Ex{{{\tau_1}}}} \cdot \left|\Ex{X_{{\tau_1}}}- \Ex{\til{X}_{{\tau_1}}}\right|.
\end{align*}
By the description of the coupling above, it is clear that given $|\cP(\tau_1)|$, the number of attempted jumps to the left by $X$ has the binomial distribution with parameters $|\cP(\tau_1)|$ and $e^{-\lambda}/(e^\lambda + e^{-\lambda})$. Using the above coupling we see that $X_{\tau_1}$ is equal to $\til{X}_{\tau_1}$ on the event that there is no left jump by time $\tau_1$. We then get by a union bound
\begin{align}\label{eq:totallyasymmetric}
\left|	\Ex{X_{{\tau_1}}}- \Ex{\til{X}_{{\tau_1}}}\right| \leq \Ex{|\cP(\tau_1)|^2 \cdot \frac{e^{-\lambda}}{e^\lambda+e^{-\lambda}}}\leq C\cdot \exp(-2\lambda), 
\end{align}
where $C$ is a positive constant using also Lemma~\ref{cor:exptailsofutauone}. We can thus deduce that as $\lambda\to \infty$ the above tends to $0$, and hence we conclude that 
\begin{equation}
\lim_{\lambda \rightarrow \infty} v(\lambda) = \bar{v},
\end{equation}
where $\bar{v}$ is given in Lemma~\ref{lem:speedinonedim}.
\end{proof}

 \subsection{Asymptotic expression for the speed}\label{sec:asymptoticexpression}

In this section we prove Proposition~\ref{pro:MonotoneSubsequence}. In order to do so, we first construct a coupling between $(X,\eta)$ and a one-dimensional biased random walk $Y$ in a suitably defined evolving environment $\xi$ on $\Z$ for which we can calculate an asymptotic expression for the speed using Lemma~\ref{lem:speedinonedim} and a time change.

In the following, with a slight abuse of notation, we identify $\Z$ with the $\eone$-axis of $\Z^d$, and for $x\in \Z$, we write $x+\eone$ for the edge $(x,x+1)$.

The purpose of this coupling is to allow us to transfer the  dynamics on $\Z^d$ to a one-dimensional system, which we can study explicitly. More precisely, we consider three Poisson processes  
$\mathcal{P}^1,\mathcal{P}^2,\mathcal{P}^3$. The processes $\mathcal{P}^1$ and  $\mathcal{P}^2$ correspond to jumps in the $ \pm \eone$ direction, while $\mathcal{P}^3$ corresponds to all other directions of $\Z^d$. At the points of $\mathcal{P}^1$ and $\mathcal{P}^2$, we let the two processes $X$ and $Y$ evolve together, and we update their respective infected sets together.
We stop the coupling either at the first jump in direction $-\eone$ or the second time a point in $\mathcal{P}^3$ occurs, and continue the processes afterwards according to the correct marginal transition rates; see below. Our key observation is that as long as the edge to the right of the walker's current location has not been examined before, we can identify its state with the state of the edge on $\Z$ with the same $\eone$-coordinate. 
We make this rigorous in the last part of the definition below.


\begin{definition}[Coupling between $(X,\eta)$ and $(Y,\xi)$]\label{def:couplingxandy}
	\rm{
	Let $\til{\mu}=\mu+p(2d-2)Z_\lambda^{-1}$, where we recall that $Z_\lambda = e^{\lambda}+ e^{-\lambda} + 2d-2$.
 Both $X$ and~ $Y$ start from~$0$. We let the environment $\eta$ evolve according to dynamical percolation on $\Z^d$ with parameters $\mu,p$ and $\eta_0\sim \pi_p$. The edges to the left of $0$ in the environment $\xi$ update according to dynamical percolation with parameters $\til{\mu},p$.

 Let $\mathcal{P}^1, \cP^2$ and $\cP^3$ be three independent Poisson processes of parameters $e^{\lambda}  Z_\lambda^{-1}, e^{-\lambda} Z_\lambda^{-1}$ and $(2d-2)Z_\lambda^{-1}$,  respectively. At the points of $\cP^1$ (resp.\ $\cP^2$) both $X$ and $Y$ attempt a jump to the right (resp.\ to the left),  and we add the corresponding edges (more precisely the lowest numbered copies not in the infected sets as in Definition~\ref{def:RegenTimes}) to their respective infected sets.  Then we say that the two (copies of) edges are a match. At the points of $\cP^3$ the walk $X$ attempts a jump in one of the $2d-2$ directions other than $\eone$ and $-\eone$ chosen uniformly at random, and we add the corresponding copy of the edge to the infected set of $X$ only. 
 
 We now explain how to remove copies of edges from the infected sets: we pick a copy from the infected set of $X$ to be removed in the same way as in Definition~\ref{def:RegenTimes} (each copy is being picked at rate $\mu$) and we also remove its match if it exists from the infected set of $Y$. We then update the corresponding edges in $\eta$ and $\xi$ in the same way as in Definition~\ref{def:RegenTimes} (i.e.\ if the copies are of the form $e_{i,1}$ for some~$i$). 
 
 Below whenever we say that we stop the coupling, afterwards we  continue $(X,\eta)$ and $(Y,\xi)$ by letting them attempt jumps at the points of $\cP^1, \cP^2$ and $\cP^3$ (the latter only for $X$) and each edge copy of $Y$ in its infected set refreshes also at the points of an additional Poisson process $\til{\cP}$ of parameter $p(2d-2)Z_\lambda^{-1}$. If a copy in the infected set of $Y$ refreshes according to this Poisson process, then we do not remove it from the infected set. However, if that copy is of the form $e_{i,1}$ for some $i$, then we update the state of its corresponding edge in $\xi$.
  
 Let $(T_i)$ be the jump times of~$\cP^3$ and let $S$ be the first point of $\cP^2$. We stop the coupling at time~$S\wedge T_2$.  For every edge $e$, we let $E(e)$ be the first time that the state of $e$ is examined by $Y$ and $C(e)$ be the first time the edge $e$ is crossed by $Y$. 
When $E(e)<T_1\wedge S$, then for times~$s\in [E(e), C(e)\wedge T_1\wedge S]$ we set $\xi_s(e)=\eta_s(e)$.  At time $T_1$, the walk $X$ attempts a jump in one of the $2d-2$ directions other than $\eone$ and $-\eone$ chosen uniformly at random. For each edge $e$ such that  $E(e)\in (T_1,T_2\wedge S)$ and for times $t\in [E(e),C(e)\wedge T_2\wedge S]$ we set $\xi_t(e)=\eta_t(X_t+\eone)$. During the time interval $(C(e)\wedge T_2\wedge S, T_2\wedge S)$, we refresh the edge $e$ in the environment $\xi$ also at the points of an additional independent Poisson process $\til{\cP}$ of parameter $p(2d-2)Z_\lambda^{-1}$. As already mentioned above, we stress again that these updates do not affect the infected set. 

We let $(\tau_i)$ be the successive times at which the infected set of $X$ becomes equal to the empty set. Then by the definition of the process $Y$ we see that also the infected set of $Y$ becomes empty at times $\tau_i$ for all $i$. 
  }
\end{definition}

 \begin{remark}
	\rm{
	We note that in the above coupling once an edge $e$ has been examined by $Y$, it then refreshes at rate $\til{\mu}$. Indeed, up until the first point of $\cP^3$, it updates at rate $\mu$. At time $T_1$, if the edge that $X$ examines is open, which happens with probability $p$, then the state of $e$ in $\xi$ is updated to the state of $X_{T_1}+\eone$, which is distributed according to $\mathrm{Ber}(p)$, as $X_{T_1}+\eone$ had not been examined before. Hence the rate at which $e$ updates is $p\cdot (2d-2)Z_\lambda^{-1}$, where the factor $p$ comes from the probability that the edge that $X$ examines at time $T_1$ is open and $(2d-2)Z_\lambda^{-1}$ is the rate of $\cP^3$. Using the sequence $(\tau_i)$ we see that the speed $v^Y$ of $Y$ is given by 
		\[
	v^Y(\lambda) = \frac{\estart{Y_{\tau_1}}{} }{\estart{\tau_1}{}}.
	\] 	
	
	}
\end{remark}

%
%

\begin{lemma}\label{pro:ApproximateLimitingSpeed} For all $p \in (0,1)$ and $\mu>0$, there exist constants $\lambda_0, c>0$, depending only on $\mu$, such that for all $\lambda \geq \lambda_0$,
\begin{equation}
\left|v^Y(\lambda)- v(\lambda)\right| \leq c \exp(-2\lambda).
\end{equation}
\end{lemma}

\begin{proof}[\bf Proof]
Recall that $S$ is the first point of $\cP^2$ and $(T_i)$ are the points of $\cP^3$. Let $A$ be the event that the coupling stops before time $\tau_1$, i.e.\
\[
A=\{S<\tau_1 \} \cup \{T_2<\tau_1\}.
\]
Then we have 
\[
|v^Y(\lambda) - v(\lambda)| \leq \frac{1}{\estart{\tau_1}{}} \cdot \estart{|X_{\tau_1}^1 - Y_{\tau_1}|\1(A)}{} 
\]
We write $\cU_a(t)$ for the total number of points of $\cP^1\cup \cP^2\cup\cP^3$ that have arrived before time~$t$. Then we obtain
\[
\Ex{|X_{\tau_1}^1 - Y_{{\tau}_1}|\1(A)} \leq 2 \, \Ex{\1(A) \cdot \cU_a(\tau_1)}.
\]
A key observation is that $\tau_1$ and $\cU_a(\tau_1)$ only depend on $\cP^1\cup \cP^2\cup \cP^3$ and the evolution of the size of the infected set, which increases at the points of the Poisson process $\cP^1\cup \cP^2\cup \cP^3$ and decreases at an independent rate~$\mu$.  
This together with the thinning property of Poisson processes yields that, conditional on $\cU_a(\tau_1)$, the numbers of points in $\cP^2[0,\tau_1]$ and $\cP^3[0,\tau_1]$ follow the binomial distribution with parameters $(\cU_a(\tau_1), e^{-\lambda}Z_\lambda^{-1})$ and $(\cU_a(\tau_1), (2d-2)Z_\lambda^{-1})$ respectively. Using this we then get 
\begin{align*}
	\Ex{\1(A) \cdot \cU_a(\tau_1)} \leq \Ex{(\cU_a(\tau_1))^2\cdot e^{-\lambda}Z_\lambda^{-1} + (\cU_a(\tau_1))^3\cdot (2d-2)^2 Z_{\lambda}^{-2}}.
\end{align*}
Since $Z_\lambda^{-1} = \cO(e^{-\lambda})$ and by Lemma~\ref{cor:exptailsofutauone}  there exists a positive constant $C_\mu$ such that $\Ex{(\cU_a(\tau_1))^3}\leq C_{\mu}<\infty$, it follows that 
\begin{align*}
	\Ex{\1(A) \cdot \cU_a(\tau_1)} \leq \cO(e^{-2\lambda})
	\end{align*}
with the implicit constants depending only on $\mu$, and hence this concludes the proof.
\end{proof}

\begin{proof}[\bf Proof of Proposition~\ref{pro:MonotoneSubsequence}]

Let $(\til{Y},\til{\xi})$ be a biased random walk on dynamical percolation on $\Z$ with parameters $\til{\mu},p$ that jumps to the right at rate $e^\lambda Z_\lambda^{-1}$ and to the left at rate $e^{-\lambda}Z_\lambda^{-1}$. Then the speed of $Y$ is the same as the speed of $\til{Y}$, since to determine it we only need to know the state of every edge after the first time the walk examines it. 

Let $\delta=\delta(\lambda)=(2d-2)Z_\lambda^{-1}$ and consider the process 
\[
(\overline{Y}_t,\overline{\xi}_t):= (\til{Y}_{t(1-\delta)^{-1}},\xi_{t(1-\delta)^{-1}}), \ \forall \ t\geq 0.
\]
Then $(\overline{Y},\overline{\xi})$ is a one-dimensional biased random walk in dynamical percolation with parameters $(p,\overline{\mu},\lambda)$, where
\[
\overline{\mu} : = \frac{\mu+ p\delta}{1-\delta}.
\]
Intuitively, $\overline{\mu}$ comes from adding an additional update rate $p\delta$ to the edges, and afterwards applying a time-change by a factor of $(1-\delta)$ to the process. 
We write $v^{\overline{Y}}$ for the speed of~$\overline{Y}$.
Let $Z$ be a random walk on dynamical percolation on $\Z$ with parameters $\overline{\mu}, p$ that only attempts jumps to the right at rate $1$. Using Lemma~\ref{lem:speedinonedim} we get that the speed of $Z$ is given by 
\[
v^Z = \frac{\overline{\mu}p}{1-p+\overline{\mu}}.
\] 
By~\eqref{eq:totallyasymmetric} we also get that for a positive constant $C'$ we have 
\begin{align*}
	\left|\frac{\overline{\mu} p }{1-p+\overline{\mu}} - {v}^{\overline{Y}}(\lambda)\right|\leq C' \exp(-2\lambda).
\end{align*}

Using now that $v^Y(\lambda) = (1-\delta)v^{\overline{Y}}(\lambda)$ and a straightforward calculation we finally conclude that for $\lambda$ sufficiently large
\begin{equation}\label{eq:SpeedRefined}
{v}^Y(\lambda) = \frac{\mu p}{1-p+\mu} -  \frac{(2d-2) p}{(1-p+\mu)^2}(\mu^2-p(1-p)) Z_{\lambda}^{-1}  + \mathcal{O}(e^{-2\lambda}),
\end{equation}
where the implicit constant in $\cO$ depends only on $\mu$ and $d$. 
This together with Lemma~\ref{pro:ApproximateLimitingSpeed} finishes the proof.
\end{proof}

\subsection{Asymptotic derivative of the speed}
\label{sec:LargeBiasLimit}

In this section we prove Lemma~\ref{lem:ApproximateDerivate} by constructing a coupling between two walks with different bias parameters. 

Let $\varepsilon>0$ and let $(X_t^{\lambda},\eta_t)_{t \geq 0 }$ and $(X_t^{\lambda+\varepsilon},\eta_t) _{t \geq 0 }$ be $\lambda$-biased  (respectively $(\lambda+\epsilon)$-biased) random walks on dynamical percolation in $\Z^d$ with parameters $\mu$ and $p$.

\begin{definition}[Coupling between $X^\lambda$ and $X^{\lambda+\epsilon}$]
	\rm{
	We start both walks from $0$ and we let them both attempt jumps at the points of a Poisson process $\cP=(\cP_t)_{t\geq 0}$ of rate 1. We also let both environments evolve together until the first very bad point defined below and afterwards we couple the environments by using the same rate $\mu$ Poisson process for the removal of copies of edges of their respective infected sets. 
	 
Whenever a clock of $\cP$ rings indicating the jump attempt at time $t$ of both walkers, we sample a random variable $U$ uniformly on $[0,1]$ and proceed as follows: 

\begin{itemize}
\item [(1)] If $U<(2d-2)/Z_{\lambda+\epsilon}$, then we let both walkers attempt a jump into one of the $2d-2$ directions different from $\eone$ and $-\eone$ chosen uniformly at random. 
\item [(2)] If $U\in [(2d-2)/Z_{\lambda+\epsilon}, (2d-2)/Z_\lambda]$, then we let the $X^\lambda$ walk attempt a jump into one of the $2d-2$ directions different from $\eone$ and $-\eone$ chosen uniformly at random, while we let the $X^{\lambda+\epsilon}$ walk attempt a jump in the $\eone$ direction. 
\item [(3)] If $U \in [(2d-2)/Z_\lambda, (2d-2)/Z_\lambda+ e^{-\lambda-\epsilon}/Z_{\lambda+\epsilon}]$, then we let both walkers attempt a jump in the $-\eone$ direction. 
\item [(4)] If $U \in [(2d-2)/Z_\lambda+e^{-\lambda-\epsilon}/Z_{\lambda+\epsilon}, 1-e^\lambda/Z_\lambda]$, then we let the $X^\lambda$ walk attempt a jump in the $-\eone$ direction, while we let the $X^{\lambda+\epsilon}$ walk attempt a jump in the $\eone$ direction.
\item [(5)] If $U>1-e^\lambda/Z_\lambda$, then we let both walkers attempt a jump in the $\eone$ direction.  
\end{itemize}
 }
\end{definition}

In the following, we let $(T_i)_{i\in \N}$ be the points of the Poisson process $(\mathcal{P}_t)_{t \geq 0}$ and we colour each point independently according to the outcome of the corresponding random variable $U$ in the above coupling. We say that a point is \textbf{good} if the corresponding random variable $U$ satisfies (5), we say that a point is \textbf{bad} if $U$ satisfies (1) or (3), and we say that a point is \textbf{very bad} if $U$ satisfies (2) or (4).  
Notice that good, bad, and very bad points are again independent Poisson point processes of intensities $q_{\textup{g}}:=e^{\lambda}Z_{\lambda}^{-1}$ for good points, $q_{\textup{b}}:=(2d-2+e^{-\lambda-\varepsilon})/Z_{\lambda+\epsilon}$ for bad points, and 
\begin{align}\label{eq:defqvb}
q_{\textup{vb}}:=  \frac{e^{\lambda+\epsilon}}{Z_{\lambda+\epsilon}} - \frac{e^{\lambda}}{Z_\lambda} >0
\end{align} 
for very bad points. Note that there exist constants $c_1,c_2,c_3>0$ and $\lambda_0>0$ so that for all $\epsilon\in (0,1)$ and $\lambda \geq \lambda_0$ we get
\begin{equation}\label{eq:qbTaylorExpansion}
|q_{\textup{b}} - (2d-2)\exp(-\lambda-\epsilon)| \leq c_1  \exp(-2\lambda)
\end{equation}
and
\begin{equation}\label{eq:qvbTaylorExpansion}
| q_{\textup{vb}} -  \varepsilon(2d-2) \exp(-\lambda) | \leq c_2 \varepsilon \exp(-2\lambda) + c_3\varepsilon^2 \exp(-\lambda).
\end{equation}
 Moreover, note that the above coupling between the two random walkers ensures that they stay together until the first very bad point and both infected sets have the same size at all times. Therefore, both processes have the same sequence of regeneration times. We let $\tau_1$ be their first regeneration time.  

We write $\cU_a(t)$ for the number of points of the Poisson process $\cP$ that have arrived by time $t$. 
Let $G$ be the event that there is no bad point up to time $\tau_1$ and for every $\ell\in \N$ let $V_\ell$ be the event that $T_\ell$ is the unique very bad point of $\cU_a(\tau_1)$. Let $R$ be the event that at the first very bad point the walk $X^\lambda$ attempts a move in one of $2d-2$ directions.  

We write for all $x\in \Z^d$
\begin{equation}
\normO{x}_1 := x \cdot \eone \, .
\end{equation}

We now explain the strategy of the proof. First we express the speed as 
\[
v(\lambda) = \lim_{\epsilon\to 0} \Ex{\tau_1}^{-1} \left( \E\Big[ \normO{X^{\lambda+\varepsilon}_{\tau_1}}_1- \normO{X^{\lambda}_{{\tau}_1}}_1\Big] + \E\Big[\normO{X^{\lambda}_{{\tau}_1}}_1\Big] \right) .
\]
We note that for each $\epsilon$ the quantity $\E\Big[ \normO{X^{\lambda+\varepsilon}_{\tau_1}}_1- \normO{X^{\lambda}_{{\tau}_1}}_1\Big]$ is non-zero if and only if there is a very bad point by time $\tau_1$. In the next lemma we show that conditional on $\cU_a(\tau_1)$ and on having a unique very bad point and no bad points up to time $\tau_1$, the expectation of $\normO{X^{\lambda+\varepsilon}_{\tau_1}}_1$ is independent of $\lambda$ and $\epsilon$. We also prove an analogous statement for $\normO{X^{\lambda}_{\tau_1}}_1$. 
Hence, the dependence on $\epsilon$ comes from the probability of the event of having a very bad point by time~$\tau_1$.

\begin{lemma}\label{lem:boundsongvellandr}
	There exists a positive constant $c=c_d$ so that the following holds. Let $p\in (0,1)$ and $\mu>0$. For all $k\in \N$ and $\ell\leq k$ we have 
	\begin{align}\label{eq:boundongcomandr}
		\prcond{G^c}{\cU_a(\tau_1)=k, V_\ell}{}\leq (k-1) \cdot q_{\mathrm{b}} \quad \text{ and } \quad \prcond{R^c}{\mathcal{U}_a(\tau_1)=k, V_\ell, G}{} \leq c e^{-\lambda}.
	\end{align}
	Moreover, there exist functions $f=f_{\mu,p}, g=g_{\mu,p}: \N\times \N\to \R_+$, which do not depend on $\lambda$ or $\varepsilon$, such that 
	\begin{align*}
		\econd{\normO{X^{\lambda+\varepsilon}_{\tau_1}}_1}{\cU_a(\tau_1)=k, V_\ell, G} = f(k,\ell)  \text{ and }
  \econd{\normO{X^{\lambda}_{\tau_1}}_1}{\cU_a(\tau_1)=k,V_\ell, G,R}=g(k,\ell).
	\end{align*}
\end{lemma}

\begin{proof}[\bf Proof]
Since the distribution of $\cU_a(\tau_1)$ is independent of the colouring of the Poisson process $\cP$, it follows that conditionally on $\cU_a(\tau_1)=k$ and $V_\ell$, every point $T_i$ for $i\leq k$ with $i\neq \ell$ has probability $q_{\mathrm{b}}$ of being a bad point. Using this together with a union bound we deduce
\[
\prcond{G^c}{\cU_a(\tau_1)=k, V_\ell}{}\leq (k-1) \cdot q_{\mathrm{b}}.
\]
Using again the independence between $\cU_a(\tau_1)$ and the colouring, we obtain 
\[
	\prcond{R^c}{\cU_a(\tau_1)=k, V_\ell, G}{} =\frac{1-e^\lambda Z_\lambda^{-1} - (2d-2)Z_\lambda^{-1}-e^{-\lambda-\epsilon}Z_{\lambda+\epsilon}^{-1}}{q_{\mathrm{vb}}}\leq ce^{-\lambda}
	\] 
	for a suitable choice of $c$, completing the proof of \eqref{eq:boundongcomandr}.
	Recall that $(T_i)$ are the points of $\cP$.
We notice that on the event $ \{\cU_a(\tau_1)=k\}\cap V_\ell\cap G\cap R$ we can write
\begin{align*}
	|X_{\tau_1}^{\lambda+\epsilon}|_1 = \sum_{i=1}^{k}\1(\eta_{T_i}(X^{\lambda+\epsilon}_{T_i-}, X^{\lambda+\epsilon}_{T_i-}+\eone)=1)
\end{align*}
and
\begin{align*}
 |X_{\tau_1}^{\lambda}|_1 = \sum_{\substack{i=1\\ i\neq \ell}}^{k}\1(\eta_{T_i}(X^\lambda_{T_i-}, X^\lambda_{T_i-}+\eone)=1).
\end{align*}
Using the independence between the Poisson process $\cP=(T_i)$ and the colouring of each point as good, bad or very bad together with the definition of the regeneration time $\tau_1$ which is independent of the colouring of the process $\cP$ (because even if we examine the same edge multiple times we still add a copy of it to the infected set), we see that  
\begin{align*}
	\cL{((T_1,\ldots, T_k), \eta\, \vert \  \cU_a(\tau_1)=k, V_\ell, G, R)}{} = \cL{((T_1,\ldots, T_k), \eta\, \vert \  \cU_a(\tau_1)=k)}{}. 
\end{align*}
In particular, this shows that the conditional law of $((T_1,\ldots, T_k), \eta)$ given $\cU_a(\tau_1)=k,  V_\ell, G, R$ is independent of~$\lambda$. We note that under this conditioning, $X^{\lambda+\epsilon}$ becomes a walk that only attempts jumps to the right at the times $T_1,\ldots, T_k$ and $X^\lambda$ attempts jumps to the right at the times $T_i$ for $i\leq k$ and $i\neq \ell$ and attempts a jump to one of $2d-2$ directions at time $T_\ell$. 
Therefore, we deduce that there exist functions $f=f_{\mu,p}$ and $g=g_{\mu,p}$ independent of $\lambda$ and $\epsilon$ so that 
\begin{align*}
	\econd{\normO{X^{\lambda+\varepsilon}_{\tau_1}}_1}{G, \ \cU_a(\tau_1)=k, V_\ell} = f(k,\ell) 
\end{align*}
and
\begin{align*}
\econd{\normO{X^{\lambda}_{\tau_1}}_1}{G, \ \cU_a(\tau_1)=k, V_\ell, R} = g(k,\ell) 
	\end{align*}
and this concludes the proof.
\end{proof}

We are now ready to prove Lemma~\ref{lem:ApproximateDerivate}. 

\begin{proof}[\bf Proof of Lemma~\ref{lem:ApproximateDerivate}]

We start the proof by recalling from Proposition~\ref{pro:RegenerationSpeed} that
\begin{equation*}
v(\lambda+\varepsilon)-v(\lambda) = \E[\tau_1]^{-1}\E\Big[ \normO{X^{\lambda+\varepsilon}_{\tau_1}}_1- \normO{X^{\lambda}_{{\tau}_1}}_1\Big].
\end{equation*}
Let $A_\epsilon$ be the event that there exists a very bad point before $\tau_1$. Then we have 
\begin{equation}\label{eq:speedequation}
v(\lambda+\varepsilon)-v(\lambda)=\E[\tau_1]^{-1}\E\Big[ \normO{X^{\lambda+\varepsilon}_{\tau_1}}_1- \normO{X^{\lambda}_{{\tau}_1}}_1  \, \Big| \, A_{\varepsilon}\Big]\P(A_{\varepsilon}) ,
\end{equation}  noting that on the complement of the event $A_{\varepsilon}$, the positions of the two walkers agree and hence the contribution to the expectation vanishes. 
Recall that $\cU_a(t)$ stands for the number of points of the Poisson process $\cP$ of rate $1$ up to time $t$. 
Since the assignment of good/bad/very bad points to the points of the Poisson process is independent of the value of $\tau_1$, we get 
\begin{align*}
	\pr{A_\epsilon} = \Ex{1-(1-q_{\textup{vb}})^{\cU_a(\tau_1)}}.
\end{align*}
Since $1-(1-q_{\textup{vb}})^{\cU_a(\tau_1)}\leq q_{\textup{vb}}\cdot \cU_a(\tau_1)$ by the dominated convergence theorem and L'H\^{o}pital's rule, recalling the approximation of $q_{\mathrm{vb}}$ from~\eqref{eq:qvbTaylorExpansion}, we obtain
\begin{align}\label{eq:limtaepsovereps}
\begin{split}
\lim_{\epsilon\to 0}	\frac{\pr{A_\epsilon}}{\epsilon} &= \Ex{\lim_{\epsilon\to 0} \frac{1}{\epsilon} \left(1-\Big(1- \varepsilon(2d-2) \exp(-\lambda) + \cO(\varepsilon e^{-2\lambda}+\varepsilon^2 e^{-\lambda})\Big)^{\cU_a(\tau_1)}\right)} \\
&= (2d-2)e^{-\lambda} \cdot \Ex{\cU_a(\tau_1)} + \mathcal{O}(e^{-2\lambda}), 
\end{split}
\end{align}
where the implicit constant only depends on $\mu$ and $d$.
We next prove that there exists a positive constant $\til{C}_{\mu,p,d}$ depending only on $\mu$, $p$ and $d$ such that 
\begin{align*}
	\lim_{\epsilon\to 0}\E\Big[ \normO{X^{\lambda+\varepsilon}_{\tau_1}}_1- \normO{X^{\lambda}_{{\tau}_1}}_1  \, \Big| \, A_{\varepsilon}\Big] = \til{C}_{\mu,p,d} + \cO( e^{-\lambda}),
\end{align*}
where the implicit constant in $\cO$ depends only on $\mu$ and $d$.
We define $\til{A}_\epsilon$ to be the event that there is a unique very bad point up to time $\tau_1$. 
First, we note that 
\begin{align}
	\prcond{\til{A}_\epsilon}{A_\epsilon}{} = \frac{\Ex{\cU_a(\tau_1)\cdot q_{\textup{vb}}\cdot (1-q_{\textup{vb}})^{\cU_a(\tau_1)-1}}}{\Ex{1-(1-q_{\textup{vb}})^{\cU_a(\tau_1)}}},
\end{align}
and using similar arguments as above we get that 
\begin{align}\label{eq:condagivenatil}
\lim_{\epsilon\to 0} \prcond{\til{A}_\epsilon}{A_\epsilon}{} = 1.
\end{align}
We now have 
\begin{align}\label{eq:xtauoneaeps}
	\econd{\normO{X^{\lambda+\varepsilon}_{\tau_1}}_1}{A_\epsilon} = \econd{\normO{X^{\lambda+\varepsilon}_{\tau_1}}_1}{\til{A}_\epsilon} \prcond{\til{A}_\epsilon}{A_\epsilon}{} +\econd{\normO{X^{\lambda+\varepsilon}_{\tau_1}}_1}{\til{A}^c_\epsilon \cap A_\epsilon} \prcond{\til{A}^c_\epsilon}{A_\epsilon}{}
\end{align}
and similarly for $X^{\lambda}_{\tau_1}$. As $|X_{\tau_1}^{\lambda+\epsilon}|\leq\cU_a(\tau_1)$, we get similarly as above
\begin{align*}
	&\limsup_{\epsilon\to 0}\econd{\normO{X^{\lambda+\varepsilon}_{\tau_1}}_1}{\til{A}^c_\epsilon \cap A_\epsilon} \leq \lim_{\epsilon \to 0} \econd{\cU_a(\tau_1)}{\til{A}^c_\epsilon \cap A_\epsilon} \\
	&=\lim_{\epsilon \to 0}  \frac{\Ex{\cU_a(\tau_1) \left(1 - (1-q_{\textup{vb}})^{\cU_a(\tau_1)} - \cU_a(\tau_1)q_{\textup{vb}}(1-q_{\textup{vb}})^{\cU_a(\tau_1)-1} \right)}}{\Ex{1 - (1-q_{\textup{vb}})^{\cU_a(\tau_1)} - \cU_a(\tau_1)q_{\textup{vb}}(1-q_{\textup{vb}})^{\cU_a(\tau_1)-1}}}.
\end{align*}
Applying the dominated convergence theorem and L'H\^{o}pital's rule and using that $q_{\textup{vb}}\to 0$ as $\epsilon\to 0$, we obtain
that this last limit is equal to 
\[
\lim_{\epsilon \to 0}  \frac{\Ex{\cU_a(\tau_1) \left(1 - (1-q_{\textup{vb}})^{\cU_a(\tau_1)} - \cU_a(\tau_1)q_{\textup{vb}}(1-q_{\textup{vb}})^{\cU_a(\tau_1)-1} \right)}}{\Ex{1 - (1-q_{\textup{vb}})^{\cU_a(\tau_1)} - \cU_a(\tau_1)q_{\textup{vb}}(1-q_{\textup{vb}})^{\cU_a(\tau_1)-1}}} = \frac{\Ex{(\cU_a(\tau_1))^2(\cU_a(\tau_1)-1)}}{\Ex{\cU_a(\tau_1)(\cU_a(\tau_1)-1)}}.
\]
Since $\Ex{\cU_a(\tau_1)}>1$ and using that $\cU_a(\tau_1)$ has exponential tails by Lemma~\ref{cor:exptailsofutauone} together with~\eqref{eq:condagivenatil} gives that the second term appearing in the sum in~\eqref{eq:xtauoneaeps} converges to $0$ as $\epsilon\to 0$. For the first expectation appearing on the right hand side of~\eqref{eq:xtauoneaeps} we have 
 \begin{align}\label{eq:conditiononatil}
	\econd{\normO{X^{\lambda+\varepsilon}_{\tau_1}}_1}{\til{A}_\epsilon}  = \sum_{k}\sum_{\ell\leq k} \econd{\normO{X^{\lambda+\varepsilon}_{\tau_1}}_1}{\cU_a(\tau_1)=k, V_\ell} \prcond{\cU_a(\tau_1)=k, V_\ell}{\til{A}_\epsilon}{}  
\end{align}
and similarly for $X^{\lambda}_{\tau_1}$. Note that for any random integrable variable $Y$ and any event $H$, 
\begin{equation}
\E[Y] = \E[Y | H ] + (\E[Y | H^{c}] - \E[Y | H] )\P(H^{c}) . 
\end{equation}
Thus, for each $k$ and $\ell\leq k$ we have 
\begin{align*}
	&\econd{\normO{X^{\lambda+\varepsilon}_{\tau_1}}_1}{\cU_a(\tau_1)=k, V_\ell} = \econd{\normO{X^{\lambda+\varepsilon}_{\tau_1}}_1}{\cU_a(\tau_1)=k, V_\ell, G}  \\&+ \left( \econd{\normO{X^{\lambda+\varepsilon}_{\tau_1}}_1}{\cU_a(\tau_1)=k, V_\ell, G^c} -\econd{\normO{X^{\lambda+\varepsilon}_{\tau_1}}_1}{\cU_a(\tau_1)=k, V_\ell, G}\right)\prcond{G^c}{\cU_a(\tau_1)=k, V_\ell}{}. 
\end{align*}
Let us remark that in the case of $X^\lambda$ we also add the event $R$ to the intersection above. 
Using again the obvious bound $\normO{X^{\lambda+\varepsilon}_{\tau_1}}_1 \leq \cU_a(\tau_1)$, all four statements of Lemma~\ref{lem:boundsongvellandr} and  equations~\eqref{eq:qbTaylorExpansion} and~\eqref{eq:qvbTaylorExpansion} we get 
\begin{align*}
	\econd{\normO{X^{\lambda+\varepsilon}_{\tau_1}}_1}{\cU_a(\tau_1)=k, V_\ell} &= f(k,\ell) + \cO(k^2\cdot e^{-\lambda}) \quad \text{ and} \\
	\econd{\normO{X^{\lambda}_{\tau_1}}_1}{\cU_a(\tau_1)=k, V_\ell} &= g(k,\ell) + \cO(k^2\cdot e^{-\lambda}),
	\end{align*}	
	where the implicit constants in the terms $\cO$ above only depend on $\mu$ and $d$.
	Inserting these back into~\eqref{eq:conditiononatil} we deduce
	\begin{align}\label{eq:deduceafterpluggingin}
	\econd{\normO{X^{\lambda+\varepsilon}_{\tau_1}}_1}{\til{A}_\epsilon}  = \sum_k \sum_{\ell\leq k}  (f(k,\ell) + \cO(k^2\cdot e^{-\lambda}))\prcond{\cU_a(\tau_1)=k, V_\ell}{\til{A}_\epsilon}{} 
	\end{align}
	and similarly for $X^\lambda$. 
	Using again the independence between $\cU_a(\tau_1)$ and the colouring, we  have 
\begin{align*}
\prcond{\cU_a(\tau_1)=k, V_\ell}{\til{A}_\epsilon}{}	 = \frac{\pr{\cU_a(\tau_1)=k} \cdot q_{\textup{vb}}\cdot (1-q_{\textup{vb}})^{k-1}}{\Ex{\cU_a(\tau_1) \cdot q_{\textup{vb}}\cdot (1-q_{\textup{vb}})^{\cU_a(\tau_1)-1}}},
\end{align*}
and hence since $q_{\textup{vb}}\to 0$ as $\epsilon \to 0$ by~\eqref{eq:qvbTaylorExpansion}, we deduce
\begin{align}\label{eq:takelimaeps}
	\lim_{\epsilon\to 0}\prcond{\cU_a(\tau_1)=k, V_\ell}{\til{A}_\epsilon}{} =\frac{\pr{\cU_a(\tau_1)=k}}{\Ex{\cU_a(\tau_1)}}. 
\end{align}
Using again the obvious bound $|X_{\tau_1}^{\lambda+\epsilon}|_1\leq \cU_a(\tau_1)$, and hence also that $f(k,\ell)\leq k$, inserting~\eqref{eq:deduceafterpluggingin} into~\eqref{eq:xtauoneaeps} and using the dominated convergence theorem we can take the limit as $\epsilon\to 0$ and use~\eqref{eq:condagivenatil} and~\eqref{eq:takelimaeps} to obtain
\begin{align*}
	\lim_{\epsilon\to 0} \econd{\normO{X^{\lambda+\varepsilon}_{\tau_1}}_1}{A_\epsilon} &= \sum_{k}\sum_{\ell\leq k} f(k,\ell) \cdot \frac{\pr{\cU_a(\tau_1)=k}}{\Ex{\cU_a(\tau_1)}} +\cO\left(e^{-\lambda}\cdot \sum_{k} k^3\cdot \frac{\pr{\cU_a(\tau_1)=k}}{\Ex{\cU_a(\tau_1)}} \right)\\
	&=\sum_{k}\sum_{\ell\leq k} f(k,\ell) \cdot \frac{\pr{\cU_a(\tau_1)=k}}{\Ex{\cU_a(\tau_1)}} + \cO(e^{-\lambda}),
	\end{align*}
where the implicit constant in $\cO$ only depends on $\mu$ and $d$ and where for the last equality we used that $\cU_a(\tau_1)$ has exponential tails by Lemma~\ref{cor:exptailsofutauone} again, and hence a finite third moment. The analogous equality holds for $X^\lambda$ with $f$ replaced by $g$. Therefore, these together with~\eqref{eq:limtaepsovereps} and~\eqref{eq:speedequation} imply that 
\begin{align*}
	\lim_{\epsilon\to 0} \frac{v(\lambda+\epsilon)-v(\lambda)}{\epsilon} 
\end{align*}
\begin{align*}= (2d-2)\cdot\frac{\Ex{\cU_a(\tau_1)}}{\Ex{\tau_1}}\cdot \sum_{k}\sum_{\ell\leq k} (f(k,\ell)-g(k,\ell)) \cdot \frac{\pr{\cU_a(\tau_1)=k}}{\Ex{\cU_a(\tau_1)}}\cdot e^{-\lambda} +\cO(e^{-2\lambda}).
\end{align*}
This now finishes the proof as $f$ and $g$ are functions that only depend on $\mu$ and $p$ and not on $\lambda$, while the implicit constant in $\cO$ depends only on $\mu$ and $d$. 
\end{proof}

\section{Strict monotonicity of the speed for large $\mu$ or  $p$ close to $1$}  \label{sec:SpeedRWTwoDim}

As already mentioned in the introduction and as we saw in Section~\ref{sec:SpeedRWOneDim}, for $d=1$, the function $\lambda \mapsto v(\lambda)$ is strictly increasing for any fixed choice of the percolation parameters $p\in (0,1]$ and $\mu>0$ due to a coupling argument. Let us emphasise that this argument cannot be extended for $d\geq 2$, as Theorem~\ref{thm:mainTwoDim} demonstrates. However, we identify in the following  two regimes of parameters $\mu$ and $p$ in $d\geq 2$ dimensions, where the speed is strictly increasing for all $\lambda>0$.

Recall the function $f(\lambda)$ from~\eqref{eq:SumRepresenationfAlpha} as well as $Z_\lambda$ from \eqref{eq:defz} and $Z'_{\lambda}=e^{\lambda}-e^{-\lambda}$, and $p_0$ from \eqref{eq:ProbabilityUnbiased}. For $k_a,\ell_a,k,\ell\in \N_0$ and $m \in \N$, we write 
\begin{equation}\label{eq:f5}
f_{k_a,\ell_a,k,\ell,m}(\lambda):=(k-\ell)e^{\lambda (k_a-\ell_a)} \left(\frac{2d}{Z_\lambda} \right)^m \left( k_a-\ell_a -m\cdot \frac{Z_\lambda'}{Z_\lambda} \right) p_0(k_a,\ell_a,k,\ell,m)
\end{equation}
and recall from~\eqref{eq:derivativeoff} that 
\[
f'(\lambda) = \sum_{m\in \N} \sum_{\substack{k_a+\ell_a \leq  m \\  k\leq k_a, \ell \leq \ell_a}} f_{k_a,\ell_a,k,\ell,m}(\lambda).
\]

\begin{proposition}\label{pro:MonotoneMuLarge} Fix $p\in (0,1)$. There exists some constant $\tilde{\mu}=\tilde{\mu}(p)>0$ such that for all $\mu>\tilde{\mu}$, we have that $\lambda \mapsto v_{\mu,p}(\lambda)$ is strictly increasing in $\lambda>0$.
\end{proposition}

\begin{proof}[\bf Proof] 

It suffices to prove that $f'(\lambda)$ is strictly positive for all $\lambda>0$ provided $\mu$ is sufficiently large.
For all $k_a,\ell_a,k,\ell \in \N_0$ and $m\geq 2$ using Lemma~\ref{cor:exptailsofutauone} we have 
\begin{equation}\label{eq:decayofrandl}
p_0(k_a,\ell_a,k,\ell,m) \leq \exp(-c_{\mu}m\big)
\end{equation}
for some constant $c_{\mu}$ with $c_{\mu}\rightarrow \infty$ as $\mu \rightarrow \infty$. 
By the construction of the regeneration time~$\tau_1$ we get
\begin{equation}\label{eq:positiveprob}
p_0(1,0,1,0,1) = p_0(0,1,0,1,1) \geq  p\cdot \frac{1}{2d}\cdot \frac{\mu}{\mu+1}\geq \frac{p}{4d} \, ,
\end{equation} for all $\mu \geq 1$. 
For every $m$ we let 
\begin{align*}
A_m:=&\big\{(k_a,\ell_a,k,\ell)\in \N_0^4: k\leq k_a \text{ and } \ell \leq \ell_a \text{ and }  k_a+\ell_a \leq m\big\} \\
&\setminus \{ (m,0,k,0),(0,m,0,\ell)  : k,\ell\leq m \} .
\end{align*}
We now get 
\begin{align*}
\sum_{m\geq 2} \sum_{(k_a,\ell_a,k,\ell)\in A_m} |f_{k_a,\ell_a,k,\ell,m}(\lambda)|
 &\leq \sum_{m\geq 2} 2m^6 \cdot e^{\lambda (m-1)} \left(\frac{2d}{Z_\lambda} \right)^{m} e^{-\c_\mu m} \\
 &\leq e^{-\lambda}\sum_{m\geq 2} 2m^6 (2d)^m e^{-\c_\mu m}. 
	\end{align*}
Since $c_\mu\to \infty$ as $\mu\to \infty$, by taking $\mu$ sufficiently large we can make the sum above as small as desired. Using~\eqref{eq:positiveprob} and that $d\geq 2$ we claim that there exists a positive constant $c$ such that for all $\lambda>0$
\begin{equation}\label{eq:FirstTerms}
f_{1,0,1,0,1}(\lambda) + f_{0,1,0,1,1}(\lambda)\geq c e^{-\lambda} . 
\end{equation}
To see this, note that for $\lambda \rightarrow 0$, the left-hand side in \eqref{eq:FirstTerms} converges to $2p_0(1,0,1,0,1)>0$. For $\lambda \rightarrow \infty$, the estimate in \eqref{eq:FirstTerms} follows from \eqref{eq:f5}. This allows to conclude \eqref{eq:FirstTerms} using continuity in $\lambda$.
We thus deduce that for $\mu$ sufficiently large we get for all $\lambda>0$
\begin{align}\label{eq:DoubleSumNew}
	\sum_{m\geq 2} \sum_{(k_a,\ell_a,k,\ell)\in A_m} |f_{k_a,\ell_a,k,\ell,m}(\lambda)| \leq \frac{1}{2} (f_{1,0,1,0,1}(\lambda) + f_{0,1,0,1,1}(\lambda)).
\end{align}
Taking now into account the summands not contained in $A_m$, note that
\begin{equation}
f_{m,0,n,0,m}(\lambda)\geq f_{0,m,0,n,m}(\lambda)
\end{equation} for all $\lambda>0$ and $m\in \N$  with $n \leq m$. Moreover, 
\begin{equation}
f_{0,m,0,n,m}(\lambda) = n m e^{-\lambda m } \left(\frac{2d}{Z_\lambda} \right)^m \left( 1+ \frac{Z_\lambda'}{Z_\lambda} \right)p_0(0,m,0,n,m)  \geq 0.
\end{equation} 
In view of \eqref{eq:DoubleSumNew} we obtain that 
\[
f'(\lambda)\geq \frac{1}{2} (f_{1,0,1,0,1}(\lambda)+f_{0,1,0,1,1}(\lambda)).
\]
Using again~\eqref{eq:positiveprob} we get that 
\begin{align*}
	f_{1,0,1,0,1}(\lambda)+f_{0,1,0,1,1}(\lambda)\geq \frac{2d}{Z_\lambda}\cdot \frac{(2d-2)(e^\lambda+e^{-\lambda})+4}{Z_\lambda} \cdot \frac{p}{4d}>0
\end{align*}
for all $\lambda>0$ and this concludes the proof.
\end{proof}

\begin{proposition}\label{pro:MonotonePLarge} Fix $\mu>0$. There exists some constant $\tilde{p}=\tilde{p}(\mu) \in (0,1)$ such that for all $p \in (\tilde{p},1)$, we have that $\lambda \mapsto v(\lambda)$ is strictly increasing in $\lambda>0$.
\end{proposition}

\begin{proof}[\bf Proof] 
Let $p$ be sufficiently close to $1$ so that $\mu^2>p(1-p)$ and let $\lambda_0=\lambda_0(\mu,d)$ be as in Theorem~\ref{thm:mainTwoDim}. For all $\lambda \geq \lambda_0$ the speed is strictly increasing by Theorem~\ref{thm:mainTwoDim}. Thus it remains to show that the speed is strictly increasing for all $\lambda\in (0,\lambda_0]$ for all $p$ close to $1$. To do this, we will prove that $v'(\lambda)>0$ for all such $\lambda$. 

In this proof we want to emphasise the dependence of the speed on the percolation parameter $p$, so we write $v(\lambda,p)=v(\lambda)$.
Observe that when $p=1$, the speed $v'(\lambda,1)\geq  c_{d} \cdot e^{-\lambda}$ for all $\lambda>0$, where $c_{d}$ is a constant depending on $d$, noting that $v(\lambda,1)=(e^{\lambda}-e^{-\lambda})Z_{\lambda}^{-1}$. It thus suffices to prove that for $p$ sufficiently close to $1$, 
\begin{align}\label{eq:goalforspeed}
|v'(\lambda,p)-v'(\lambda,1)| \leq \frac{c_{d}}{2}\cdot e^{-\lambda} 
\end{align}
uniformly for all $\lambda\in (0, \lambda_0]$. Recall the expression for $v'(\lambda,p)$ from Lemma~\ref{lem:Differentiability}. We want to compare $\estart{X_{\tau_1}^1(R_a-L_a)}{\lambda,p}$ to $\estart{X_{\tau_1}^1(R_a-L_a)}{\lambda,1}$ and also $\estart{X_{\tau_1}^1\cdot \cU_a(\tau_1)}{\lambda,p}$ to $\estart{X_{\tau_1}^1\cdot \cU_a(\tau_1)}{\lambda,1}$. To do this, we couple the walks in the two environments by letting them attempt the same jumps at  the times $(T_i)$ of a Poisson process of rate $1$ and using the same Poisson process of rate $\mu$ to remove copies of edges from their  infected sets. We let~$\tau_1$ be their first regeneration time. We now let $\kappa$ be the index of the first jump time  when the walk in the $p$-dynamical percolation process attempts a jump along a closed edge. Then up until time  $T_\kappa$ the two walks are in the same location. Note that $\kappa$ stochastically dominates a geometric random variable of parameter~$1-p$, because for all $s<t$ and all edges $e$ we have 
\[
\prcond{e\text{ is open at time } t}{e\text{ is open at time } s}{p} \geq p.
\]
Writing $\cU_a(t)$ for the number of attempted jumps in the interval $[0,t]$, by a union bound we get 
\begin{align*}
	\pr{\kappa<\cU_a(\tau_1)} \leq \pr{\cU_a(\tau_1)>\frac{1}{\sqrt{1-p}}} + \pr{\kappa<\frac{1}{\sqrt{1-p}}} \leq C_\mu \cdot \sqrt{1-p},
\end{align*}
where $C_\mu$ is a constant depending on $\mu$. Using this and the bound $|X_{\tau_1}^1|\leq \cU_a(\tau_1)$  we get 
\begin{align*}
|\estart{X_{\tau_1}^1(R_a-L_a)}{\lambda,p}  - \estart{X_{\tau_1}^1(R_a-L_a)}{\lambda,1}| \leq 2 \estart{(\cU_a(\tau_1))^2\1(\kappa<\cU_a(\tau_1))}{\lambda,p}. 	
\end{align*}
By the Cauchy-Schwarz inequality and the exponential tails of $\cU_a(\tau_1)$ uniformly in $p$ by Lemma~\ref{cor:exptailsofutauone}, we deduce
\begin{align*}
\estart{(\cU_a(\tau_1))^2\1(\kappa<\cU_a(\tau_1))}{\lambda,p} \leq C_{\mu}' (1-p)^{1/4}.
\end{align*}
Similarly we can bound the remaining terms appearing in $v'(\lambda)$. Recalling that $\lambda_0$ depends only on $\mu$ and $d$, 
taking $p=p(\lambda_0,\mu)$ sufficiently close to $1$, we get~\eqref{eq:goalforspeed} and this concludes the proof.
\end{proof}

%
%

\begin{acks}[Acknowledgments]
 We thank Frank den Hollander and Remco van der Hofstad for valuable discussions.   We also thank the anonymous referees for valuable comments and suggestions.  Moreover,  we are  grateful to Jan Nagel for pointing out a mistake in an earlier version. 
\end{acks}
\begin{funding}
 DS acknowledges the DAAD PRIME program for financial support.
\end{funding}



\bibliographystyle{imsart-number} 
\bibliography{BRWonDP}       


\end{document}